\documentclass[oneside,open=any, abstract=on, paper=a4, fontsize=11pt]{scrartcl}
\pdfoutput=1

\usepackage[a4paper,top=35mm,bottom=30mm,left=25mm,right=25mm]{geometry}
\usepackage[english]{babel}
\usepackage[utf8]{inputenc}

\usepackage{xargs} 
\usepackage[pdftex,dvipsnames]{xcolor}  
\usepackage{graphicx} 
\usepackage{amssymb}
\usepackage{amsmath}
\usepackage{amsthm} 
\usepackage{mathtools} 
\usepackage{bm} 
\usepackage[super]{nth} 
\usepackage[sort, numbers, square]{natbib} 
\usepackage{physics} 
\usepackage{enumitem} 
\usepackage{tabularx} 
\usepackage{caption} 
\usepackage[position=b]{subcaption} 
\usepackage{tikz} 
\usepackage{tikz-cd}
\usepackage{tkz-euclide} 
\usepackage{pgfplots} 

\usepackage{csquotes} 
\usepackage{aliascnt} 
\usepackage[outline]{contour} 

\usepackage{bbm} 
\usepackage{tikz-cd} 
\usepackage[activate=true,final,tracking=true,kerning=true,factor=1100,stretch=10,shrink=10]{microtype} 
\usepackage{xr-hyper} 
\usepackage[pdfauthor={},hidelinks]{hyperref}
\usepackage[ruled,vlined]{algorithm2e}
\usepackage[capitalize,noabbrev]{cleveref} 
\usepackage{autonum} 
\usepackage{xspace}
\usepackage{svg}

\makeatletter
\DeclareRobustCommand\onedot{\futurelet\@let@token\@onedot}
\def\@onedot{\ifx\@let@token.\else.\null\fi\xspace}

\MakeOuterQuote{"}

\usetikzlibrary{positioning, arrows, matrix, decorations.pathmorphing, calc}

\addto\extrasenglish{

}

\setenumerate{label=(\arabic*),itemsep=0mm} 

\newcommand{\cX}{\mathcal{X}}
\newcommand{\cY}{\mathcal{Y}}
\newcommand{\R}{\mathbb{R}}
\newcommand{\N}{\mathbb{N}}

\newcommand{\E}{\mathbb{E}}

\newcommand{\cR}{\mathcal{R}}
\newcommand{\cO}{\mathcal{O}}
\newcommand{\cH}{\mathcal{H}}
\newcommand{\cF}{\mathcal{F}}
\newcommand{\cS}{\mathcal{S}}
\newcommand{\cM}{\mathcal{M}}

\newcommand{\cV}{\mathcal{V}}
\newcommand{\cT}{\mathcal{T}}
\newcommand{\cL}{\ensuremath{\mathcal{L}}}

\newcommand{\QXY}{\ensuremath{\mu}}
\newcommand{\QX}{\ensuremath{\QXY_\cX}}

\renewcommand{\epsilon}{\varepsilon}
\renewcommand{\phi}{\varphi}
\renewcommand{\grad}{\nabla}
\renewcommand{\Im}{\operatorname{Im}}
\renewcommand{\vec}{\ensuremath{\mathrm{vec}}}

\DeclareMathOperator*{\rgrad}{grad}
\DeclareMathOperator*{\ngrad}{ngrad}

\DeclareMathOperator*{\argmin}{arg\,min}

\newcommand{\T}{\mathrm{T}}
\newcommand{\rR}{\mathrm{R}}

\newcommand{\inner}[2]{\langle #1, #2 \rangle}

\DeclareMathOperator{\diag}{diag}

\theoremstyle{definition}
\newtheorem{definition}{Definition}[section]
\newtheorem{example}[definition]{Example}
\newtheorem*{example*}{Example}

\newtheorem{remark*}{Remark}
\newtheorem{remark}[definition]{Remark}
\theoremstyle{theorem}
\newtheorem{lemma}[definition]{Lemma}
\newtheorem*{lemma*}{Lemma}
\newtheorem{corollary}[definition]{Corollary}

\newtheorem{proposition}[definition]{Proposition}

\DeclareMathOperator{\Span}{span}

\def\wrt{w.r.t\onedot}
\def\eg{e.g\onedot}
\def\ie{i.e\onedot}

\numberwithin{equation}{section}

\title{Natural Riemannian gradient for learning functional tensor networks}
\author{
	Nikolas Klug\thanks{Institute of Mathematics, University of Augsburg, 86159 Augsburg, Germany} 
	\qquad Michael Ulbrich\thanks{Department of Mathematics, Technical University of Munich, 85748 Garching b.~München, Germany} \\
	Andr\'e Uschmajew\thanks{Institute of Mathematics \& Centre for Advanced Analytics and Predictive Sciences, University of Augsburg, 86159 Augsburg, Germany}  
	\qquad Marius Willner${}^\ast$
	}
\date{\vspace{0mm}}

\begin{document}
	\maketitle	
\begin{abstract}
	We consider machine learning tasks with low-rank functional tree tensor networks (TTN) as the learning model.
	While in the case of least-squares regression, low-rank functional TTNs can be efficiently optimized using alternating optimization, this is not directly possible in other problems, such as multinomial logistic regression.
	We propose a natural Riemannian gradient descent type approach applicable to arbitrary losses which is based on the natural gradient by Amari.
	In particular, the search direction obtained by the natural gradient is independent of the choice of basis of the underlying functional tensor product space.
	Our framework applies to both the factorized and manifold-based approach for representing the functional TTN.
	For practical application, we propose a hierarchy of efficient approximations to the true natural Riemannian gradient for computing the updates in the parameter space.
	Numerical experiments confirm our theoretical findings on common classification datasets and show that using natural Riemannian gradient descent for learning considerably improves convergence behavior when compared to standard Riemannian gradient methods.
\end{abstract}

\section{Introduction}
	Many machine learning methods are based on (empirical) risk minimization.
	Let $\cX \subseteq \R^{d_x}$, $\cY \subseteq \R ^{d_y}$, $\QXY$ be a joint probability measure on $\cX \times \cY$ and $\cH$ be a set of hypotheses, also called the \emph{learning model}. The goal is to find $h \in \cH$ which minimizes the \emph{risk} $\mathcal R$, defined as
	\begin{equation}\label{eq:risk}
		\cR: \cH \to \R\ , \quad
		\cR(h) \coloneqq \E_{(x, y) \sim \QXY}[\ell(h, x, y)]
		= \int_{\cX \times \cY}  \ell(h, x, y) \QXY(dx, dy) \ .
	\end{equation}
	Here, $\ell: \cH \times \cX \times \cY \to \R$ is called \emph{loss function} and must be sufficiently regular, that is, $\ell(h, \cdot, \cdot) \in L_1(\cX \times \cY, \QXY)$ must hold for all $h \in \cH$.
	Usually, $\ell$ takes only nonnegative values.
	In this work we focus on the case where $\cH$ is a (finite-dimensional) real differentiable Riemannian manifold, that is, we consider the problem
	\begin{equation}\label{eq:problem}
		\text{Find }\quad h^\ast \in \argmin_{h \in \cH} \cR(h) \ .
	\end{equation}
	In practice, the manifold $\cH$ is usually accessed through a parametrization $F: \cM \to \cH$, where~$\cM$ is another, more tractable (finite-dimensional) Riemannian manifold.
	The parametrization $F$ is usually not unique; there can be many different parametrizations and a ``bad'' choice can severely impact the behavior of optimization algorithms, in particular first-order methods based on gradient descent.
	A popular approach to address the influence of the parametrization $F$ is the concept of the natural gradient~\cite{amari98}.

	Several common classes of learning models $\mathcal{H}$ are used in machine learning.
	Neural networks in various flavors belong to the most prominent examples and, depending on the problem, have proven to be quite successful.
	In this work, we consider a different class of learning models based on low-rank functional tree tensor networks (TTNs).
	These models represent functions $f: \Omega \subseteq \R^d \to \R^{n_0}$ in the form
	\begin{equation}\label{eq:hypo-class}
		f(x) = \inner{\mathsf{A}}{\Phi(x)}_{1, \ldots, d} \ ,
	\end{equation}
	where $\mathsf{A} \in \R^{n_0 \times \cdots \times n_d}$ is an order-$(d+1)$ tensor, $\Phi: \Omega \subseteq \R^{d} \to \R^{n_1 \times \cdots \times n_d}$ is a feature map corresponding to point evaluations in a tensor product basis (see \cref{sec: functional TN model}), and $\inner{\cdot}{\cdot}_{1, \ldots, d}$ denotes tensor contractions along the indices $1, \ldots, d$.
	Because of high-dimensionality one can usually not allow arbitrary coefficient tensors $\mathsf{A}$ in practice. 
	In low-rank models one hence restricts $\mathsf{A}$ to tensors in certain low-rank tensor decompositions which can be efficiently stored in memory and, moreover, allow an efficient computation of tensor contractions.
	In this work, we consider tensors contained in fixed-rank tree tensor network manifolds $\cM$, treated in this work mostly as the quotient manifold \wrt a multilinear parametrization.
	Note that, the learning model given by~\eqref{eq:hypo-class} is fairly simple in the sense that it is linear in the parameter tensor $\mathsf{A}$, but nonetheless possesses high expressivity because it parameterizes functions in high-dimensional tensor product spaces, depending on the choice of $\Phi$.
	By restricting the coefficient tensor $\mathsf{A}$ to a low-rank manifold, the learning model becomes nonlinear. 
	While the effect of this restriction to the expressivity is in general difficult to assess rigorously, functional tensor models can offer a more systematic approach to machine learning problems because the overall mathematical theory for tensor methods is already well-developed.

	Restricting the tensor $\mathsf{A}$ in~\eqref{eq:hypo-class} to a fixed-rank manifold $\cM$ naturally yields a parametrization $F : \cM \to \cH$ of the functional tensor network (FTN) learning model.
	In this work we show how to efficiently compute natural gradients for such parameterized low-rank FTN models.
	We develop the theory along two typical machine learning problems: least-squares regression and classification via multinomial logistic regression. 
	In (least-squares) regression, on usually assumes a functional relationship between $\cX$ and $\cY$, which means that for each $x \in \cX$, there is a unique $y = y(x)$.
	The hypotheses are functions $h: \cX \to \cY$ and the goal is to find
	\begin{equation}
		 h^\ast \in \argmin_{h \in \cH} \E_{x \sim \QX}\left[\norm{h(x) - y(x)}_\cY^2\right] \ ,
	\end{equation}
	where $\QX$ is the marginal probability measure of $\QXY$ on $\cX$.
	For such problems, a natural gradient descent approach for compositional functional tensor trains was recently proposed in~\cite{eigel25}.
	
	In classification, the vectors in $\cX$ are to be classified into one of $n_0$ classes.
	For simplicity we assume unique true labels, that is, for every $x \in \cX$ there is again a unique $y = y(x)$ (although our approach is also applicable to the more general case).
	In multinomial logistic regression (also known as softmax regression), the loss function is the negative log-likelihood of a categorical distribution, which results in the problem to find
	\begin{equation}
		 h^\ast \in \argmin_{h \in \cH} \E_{x \sim \QX}\left[-\sum_{j = 1}^{n_0} y(x)_j \log(h(x)_j)\right] \ ,
	\end{equation}
	where $y(x) \in \{0, 1\}^{n_0}$ with $y(x)_j = 1$ iff $x$ belongs to class $j$.
	
	Note that whereas in least-squares regression, standard low-rank tensor algorithms based on alternating least-squares optimization are possible and are likely to outperform gradient based methods, this is no longer the case for multinomial logistic regression:
	Here, the subproblems for the individual cores are not linear least-squares problems anymore, necessitating different algorithms such as (Riemannian) gradient descent.

	\subsubsection*{Contributions}

	In this work, we investigate how natural gradient descent can be efficiently applied to machine learning tasks with low-rank FTNs as the learning model.
	In order to rigorously account for the fact, that the parametrization of the learning model $\cH$ can itself be defined on a manifold $\cM$, we first present a self-contained derivation of the natural gradient descent algorithm in a Riemannian framework.
	The idea behind this is to leave some flexibility regarding the actual optimization methods used the parameter manifold $\cM$, explicitly enabling the tools from Riemannian optimization~\cite{absil2008,boumal23}.
	In particular, while we consider fixed-rank tree tensor networks through the quotient manifold $\cM$ in the space of tensor factors, one could in principle also consider the embedded manifolds in tensor space(see \cref{sec: functional TN model}).
	We therefore adopt the terminology of \emph{natural Riemannian gradient} throughout this work.

	Within our framework, we then derive formulas for computing and approximating the natural Riemannian gradient both for least-squares regression and the multinomial logistic regression setting for classification.
	In the context of (low-rank) FTNs computing the natural Riemannian gradient poses a central challenge, since one has to solve a linear system (see~\cref{eq:ngrad}) which can be extremely large even for moderately sized models.
	To address this problem, we propose a hierarchy of heuristic approximations, which balance approximation quality and computational cost.
	At their core, our heuristics are driven by a block-diagonalization of the linear system, leveraging the multilinear parametrization of the underlying FTN.
	The resulting algorithms are applicable to the deterministic learning setting, where all samples can be treated in one batch.
	In addition, we propose an algorithm for stochastic natural Riemannian gradient descent, where further consideration is required to obtain stable estimates for the natural gradient.
	While the stochastic version of our algorithm is currently based on an ad hoc approach and not systematically developed, we consider it an additional contribution of our work.

	In the numerical experiments we compare the proposed algorithms with standard Riemannian gradient descent in practice.
	For least-squares regression, we conduct tests for a recovery problem on a toy dataset and validate that the choice of the tensor product basis significantly influences the convergence rate for standard Riemannian gradient descent.
	In comparison, our natural Riemannian gradient approach reduces this effect considerably and requires a fewer number iterations, although in its deterministic version these are more expensive. For experiments in a more realistic setting, we test our deterministic and stochastic algorithms on two standard classification datasets, \texttt{digits} and \texttt{MNIST}.
	Again, we observe that the algorithms based on natural Riemannian gradient lead to considerably faster convergence \wrt the number of iterations.
	However, even with approximations, natural Riemannian gradient descent still comes at a higher computational cost per iterations than standard Riemannian gradient descent.
	Despite this, we can achieve faster convergence \wrt absolute time.
	Moreover, in the stochastic regime, our proposed algorithm not only converges faster but also results in improved test accuracy.

	\subsubsection*{Related work}

	The concept of a natural gradient as a descent direction based on the (Riemannian) geometry of statistical learning models was introduced by \citet{amari98} and is nowadays well known.
	We also refer to the more recent survey by \citet{martens20} on the natural gradient, and the considerations about K-FAC \cite{martens15}, which inspired parts of the algorithmic design in this work.
	Further concerning algorithmics, one of our fully-diagonal approximation schemes is related but not equivalent to the Fisher ADAM algorithm used in \cite{hwang24}.
	Recently, a formula for evaluating natural Riemannian gradients was proposed by \cite{hu24}, who justify their method using tools from information geometry in a bottom-up approach.
	This is different from our top-down treatment, where the natural gradient arises naturally from functional considerations.
	
	Functional low-rank tensor formats have been initially proposed for high-dimensional applications in quantum chemistry and PDEs; see the survey articles~\cite{Grasedyck13,Bachmayr16,Bachmayr23} and monographs~\cite{Hackbusch19,Khoromskij18}.
	In particular, tree tensor networks come in various flavors, the notable examples being the Tucker format, the tensor train (TT) format~\cite{oseledets11} and the hierarchical Tucker (HT) format~\cite{HackbuschKuehn09}.

	Beginning with seminal works such as \cite{Stoudenmire2016,Novikov18}, low-rank functional tensor networks have received increasing attention for machine learning and are still an active field of research.
	The work \cite{Stoudenmire2016}, similar as several subsequent ones, e.g.~\cite{Chen2018,gorodetsky18,Klus2019}, followed an alternating optimization approach similar to DMRG type algorithms for quantum systems.
	Notably, for the special case of the least-squares loss, subproblems can be solved optimally.
	While this is not possible for, e.g.~classification via logistic regression, they still can be treated via alternating optimization by applying nonlinear solvers to subproblems~\cite{Chen2018}.
	Recently, another approach has been taken in \cite{yamauchi25} by via an expectation-maximization alternating least squares algorithm.

	In contrast, \cite{Novikov18} directly employed Riemannian gradient descent on the fixed-rank TT manifold.
	Such algorithms are based on the manifold properties of fixed-rank tree tensor networks as established in~\cite{UschmajewVandereycken13} and had been initially successfully applied for tensor completion problems~\cite{Kressner14,SilvaHerrmann15,steinlechner16} and also eigenvalue problems~\cite{Rakhuba19}.
	The more recent work by \citet{willner25} provides a framework for Riemannian gradient descent based on the quotient manifold formalism for functional tensor networks, on which the algorithmic of this work is based.

	Among the mentioned references, let us particularly highlight the work by \citet{SilvaHerrmann15}, where for the problem of tensor completion in the hierarchical Tucker format a Gauss-Newton-based algorithm is suggested based on a quotient manifold formalism.
	This shares similarity with our approach in the sense that it aims at undoing the reparametrization effect.
	However, their work does not consider functional tensor networks and therefore does not capture the natural gradient in a statistical sense.
	The treatment of FTNs is more involved because of the extra layer of complexity added by the feature map $\Phi$ in \cref{eq:hypo-class}; see also the discussion at the end of \cref{sec:least_squares_ftn}.
	In addition, the empirical approximation with samples requires further considerations which we develop in this work.

	There is also recent interest in compositional FTNs~\cite{schneideroster24}.
	These models compose several functional tensor networks (\eg functional tensor trains) as layers, similar to neural networks, to form one larger model.
	For compositional functional tensor networks, natural gradient descent was considered in the recent work \cite{eigel25}, where the authors compute the natural gradient on the embedded manifold of low-rank tensor coefficients for least-squares problems.
	Compared to~\cite{eigel25}, we do not consider composition of tensor networks, but solely focus on the core problem of minimizing a function on a functional low-rank tensor manifold from a Riemannian perspective.
	Moreover, we compute the natural gradient directly through the multilinear parametrization of the tensor network, which can usually be implemented more conveniently.

	\subsubsection*{Outline}

	The work is structured as follows. In \cref{sec: Natural Riemannian gradient descent} we present a self-contained derivation of the natural Riemannian gradient descent algorithm for empirical risk minimization and its empirical counter part. The main formulas are \eqref{eq:ngrad} and \eqref{eq:empirical linear equation}, respectively, describing the Gauss--Newton type linear systems that need to be solved for computing the natural Riemannian gradient.

	In \cref{sec: Natural Gradient for functional tensor network manifolds} we then apply the concepts to optimization on functional low-rank tensor manifolds. 
	These manifolds are introduced in \cref{sec: functional TN model} and \cref{sec:low-rank-model}. As guiding examples the TT format and the balanced binary tree format are discussed. Notably, our representation of these manifolds is a quotient structure in the space of core tensors (see \cref{rem: quotient formalism}). An interesting aspect of functional tensor models is the choice of basis in the tensor product space, which has an interesting parallel interpretation as feature map into rank-one tensors. Due to an invariance of low-rank tensors under change of tensor product, the particular choice does not affect the learning model but may still influences practical computations as discussed in \cref{sec:low-rank-model}. In \cref{sec:least_squares_ftn} and \cref{sec:logistic-regression} the computation of natural Riemannian gradient is discussed separately for least-squares regression and multinomial regression, respectively. \Cref{sec:approximations} then presents our main practical contributions including block diagonal approximation of the Gauß-Newton system, a heuristic stochastic version natural Riemannian gradient descent employing stabilizing momentum transport and further approximation of block diagonals by scaled identities.

	The numerical experiments are presented in \cref{sec:experiments}. 
	They include a more conceptual study for least-squares recovery of an exact FTN model via in \cref{sec:Least-squares recovery problem}, in which we also inspect the influence of the basis choice. 
	The more challenging classification problems via multinomial logistic regression are treated in \cref{sec:deterministic-logistic-regression} where we also employ the full hierarchy of proposed approximations.

\section{Natural Riemannian gradient descent}\label{sec: Natural Riemannian gradient descent}

	The concept of a natural gradient for parametric optimization has been introduced by Amari~\cite{amari98} and is well understood, in particular in the context of statistical models.
	Mathematically, it shares strong similarities with the idea of Gauß-Newton methods.
	Nevertheless, in order to clearly work out some often omitted details, we present here a rather self-contained derivation from a Riemannian perspective where we view the natural gradient purely as a parametrization-independent descent direction for the empirical risk optimization problem~\eqref{eq:problem}.
	For this, we assume that $\cH$ is a \emph{finite-dimensional} real differentiable manifold with a Riemannian structure.
	This assumption reflects the practical situation that learning models are described by finitely many parameters.
	We denote the tangent space at a point $h \in \cH$ by $T_h \cH$ and the Riemannian metric on $T_h \cH$ by $\inner{\cdot}{\cdot}_h$.
	If $\cH'$ is another real differentiable Riemannian manifold and $f: \cH \to \cH'$ is a differentiable function, we write $Df(x)[\zeta]$ for the derivative of $f$ at the point $x$ in direction $\zeta$, that is, $Df(x): T_x \cH \to T_{f(x)} \cH'$ is the differential of $f$ at $x$.
	From now on we assume that all objects (measures, manifolds, functions, loss etc.) are sufficiently smooth.

	\subsection{Natural Riemannian gradient}\label{sec: natural Riemannian gradient}
	
	In principle, the considerations in this subsection apply to an arbitrary smooth map $\cR : \cH \to \R$.
	Let $\rR_h: T_h \cH \to \cH$ be a retraction for the manifold $\mathcal H$ at a fixed element $h$ (see~\cite{boumal23} for the definition).
	We first consider the linearization of the map $\cR \circ \rR_h: T_h \cH \to \R$.
	Let $\zeta \in T_h \cH$, then
	\begin{align}
		\cR(\rR_h(\zeta))
		&= \cR(\rR_h(0)) + D\cR(\rR_h(0))[D \rR_h(0)[\zeta]] + \cO(\norm{\zeta}^2) \\
		&= \cR(h) + D\cR(\rR_h(0))[\zeta] + \cO(\norm{\zeta}^2) \\
		&= \cR(h) + D\cR(h)[\zeta] + \cO(\norm{\zeta}^2) \ .
	\end{align}
	Similarly to the vector space case, we can obtain a (Riemannian) gradient $\rgrad \cR(h)$ as the Riesz-representative of $D \cR(h)$ with respect to the Riemannian metric $\inner{\cdot}{\cdot}_h$ at $h$, that is
	\begin{equation}
		D \cR(h)[\zeta]
		= \inner{\rgrad \cR(h)}{\zeta}_h \ .
	\end{equation}
	Specifically, for the risk $\cR$ in~\eqref{eq:risk}, the Riemannian gradient at $h$ is given by
	\begin{equation}
		\rgrad \cR(h)
		= \rgrad (\E_{(x, y) \sim \QXY}[\ell_{x, y}])(h)
		= \E_{(x, y) \sim \QXY}[\rgrad \ell_{x, y}(h)]
	\end{equation}
	where $\ell_{x, y}: \cH \to \R$ denotes the map $h \mapsto \ell(h, x, y)$. Here we deliberately interchanged the gradient with the expectation based on our assumption that all objects are sufficiently smooth. Note that different Riemannian metrics can result in different Riemannian gradients.

	The negative Riemannian gradient $- \grad \cR(h)$ at a point $h$ provides the direction of steepest descent for the function $\cR$ with respect to the geometry of the manifold $\cH$. 
	A common difficulty in practice is that the elements of $\cH$ can only be accessed through a parametrization $h = F(\Theta)$, where $F: \cM \to \cH$ and $\cM$ is another Riemannian manifold (possibly a linear space) and $\Theta \in \cM$ are the parameters of the model.
	Hence, for practical implementation of optimization algorithms on $\cH$, we need to express all objects in terms of $\Theta$ and $F$.
	In the following, we always assume that the parametrization $F$ is surjective (allowing for overparametrization) and a submersion, that is, $DF(\Theta)$ is surjective for all $\Theta$.

	Let $h = F(\Theta)$.
	All tangent vectors in $T_h \cH$ are expressed via $DF(\Theta)[\zeta]$ for some $\zeta \in T_\Theta \cM$.
	Therefore, if we wish to express the Riemannian gradient $\grad \cR(h)$ in the parameter space, we need to solve the following problem:
	\begin{equation}\label{eq:param-tangent-vector}
		\text{Find }\quad 
		\zeta \in T_\Theta \cM 
		\quad\text{~s.t.~}\quad 
		DF(\Theta)[\zeta] = \rgrad \mathcal{R}(F(\Theta)) \ .
	\end{equation}
	This problem is well-defined since $F$ is a submersion, although $\zeta$ need not be unique.
	On the other hand, we usually do not have explicit access to $\rgrad \mathcal{R}(h) = \rgrad \mathcal{R}(F(\Theta))$ but can only compute $\rgrad (\mathcal{R} \circ F) (\Theta)$.
	Interestingly, the solution of the above problem can be expressed in terms of $\rgrad (\mathcal{R} \circ F) (\Theta)$ only.
	\begin{lemma}
		Assume $DF(\Theta) : T_\Theta \cM \to T_{F(\Theta)}\cH$ is surjective. A vector $\zeta^* \in T_\Theta \cM$ solves~\eqref{eq:param-tangent-vector} if and only if it satisfies
		\begin{equation}\label{eq:ngrad}
			DF(\Theta)^\ast DF(\Theta)[\zeta^\ast] = \rgrad (\cR \circ F) (\Theta) \
		\end{equation}
		where $DF(\Theta)^\ast$ denotes the adjoint of $DF(\Theta) \colon T_{\Theta} \cM \to T_{F(\Theta)} \cH$ with respect to the corresponding metrics.
		A particular solution is
		\begin{equation}\label{eq:ngrad-mp-solution}
			\zeta^\ast = (DF(\Theta)^\ast DF(\Theta))^+[\rgrad (\mathcal{R} \circ F) (\Theta)] \ ,
		\end{equation}
		where $(DF(\Theta)^\ast DF(\Theta))^+$ denotes the Moore-Penrose inverse.
	\end{lemma}
	\begin{proof}
		The chain rule gives
		\begin{equation}
			\rgrad (\mathcal{R} \circ F) (\Theta) = D F(\Theta)^\ast [\rgrad \mathcal{R}(F(\Theta))]\ .
		\end{equation}
		Therefore, by applying $D F(\Theta)^\ast$ to both sides of~\eqref{eq:param-tangent-vector} we obtain equation~\eqref{eq:ngrad}.
		On the other hand, when applying $(D F(\Theta)^\ast)^+$ (the Moore-Penrose inverse of $D F(\Theta)^\ast$) to both sides of~\eqref{eq:ngrad}, we obtain~\eqref{eq:param-tangent-vector} because $(D F(\Theta)^\ast)^+ D F(\Theta)^\ast = D F(\Theta) D F(\Theta)^+$ is the identity on $T_{F(\Theta)} \cH = \Im DF(\Theta)$.
		Using this, we can also see that $\zeta^*$ in~\eqref{eq:ngrad-mp-solution} satisfies
		\begin{align}
			DF(\Theta)[\zeta^\ast]
			&= DF(\Theta) (DF(\Theta)^\ast DF(\Theta))^+ DF(\Theta)^\ast [\rgrad \mathcal{R}(F(\Theta))] \\
			&= DF(\Theta) DF(\Theta)^+ [\rgrad \mathcal{R}(F(\Theta))] \\
			&= \rgrad \mathcal{R}(F(\Theta)) \ ,
		\end{align}
		and hence solves~\eqref{eq:param-tangent-vector}.
	\end{proof}
	
	\begin{definition}\label{def: natural Riemannian gradient}
		A vector $\zeta^\ast \in T_\Theta \cM$ is called \emph{natural Riemannian gradient} for $\mathcal{R} : \cH \to \R$ \wrt the parametrization $F: \cM \to \cH$ at $\Theta \in \cM$ if it satisfies~\eqref{eq:ngrad}.
		Any such vector is denoted by
		\begin{equation}
			 \ngrad(\cR \circ F) \coloneqq \zeta^\ast \ .
		\end{equation}
	\end{definition}
	This notation might be considered slightly abusive, since by this definition, the natural Riemannian gradient is in general not unique.
	However, whenever we use it, we will silently assume that a particular vector has been picked. 
	The reason why we did not use, \eg,~\eqref{eq:ngrad-mp-solution} for obtaining a unique definition is that in practice we solve the linear system~\eqref{eq:param-tangent-vector} so we do not know in advance which solution will be found.
	A sufficient condition for uniqueness is that $F$ is a local diffeomorphism.
	However, this is not the case in the settings considered in this work, where the manifolds are parameterized by a multilinear map with inherent (scaling) indeterminacy.
	In some cases, if the parametrization $F$ has a special structure, the natural gradient is identical to the standard Riemannian gradient.
	\begin{corollary}
		\label{cor:ngrad-is-rgrad}
		If $DF(\Theta): T_\Theta \cM \to T_{F(\Theta)} \cH$ is isometric, that is, $\inner{DF(\Theta)[\zeta]}{DF(\Theta)[\xi]}_{F(\Theta)} = \inner{\zeta}{\xi}_\Theta$, then the natural Riemannian gradient is unique and
		\begin{equation}
			\ngrad (\cR \circ F)(\Theta) = \rgrad (\cR \circ F) (\Theta) \ .
		\end{equation}
	\end{corollary}
	In \cref{sec:logistic-regression} we briefly show how our definition of the natural Riemannian gradient coincides with that by~\citet{amari98} for a statistical setting.
	Let us also note that in case $F$ is not surjective, a natural gradient can still be defined as the solution to the least-squares problem
	\begin{equation}
		\min_{\zeta \in T_\Theta} \norm{DF(\Theta)[\zeta] - \rgrad \cR(F(\Theta))}_{F(\Theta)}^2 \ ,
	\end{equation}
	which is the approach taken in \cite{eigel25}.
	This will lead to the same formula as in \cref{eq:ngrad-mp-solution}.
	Since we always assume surjectivity, we use directly the formula~\eqref{eq:ngrad}.
	
	In case $\cM$ is a manifold embedded in a finite-dimensional Euclidean ambient space~$\cV$ (in particular~$\cM$ is then equipped with the Riemannian metric inherited from $\cV$), we can compute a natural Riemannian gradient by
	\begin{align}
		\ngrad (\cR \circ F)(\Theta)
		&= (DF(\Theta)^\ast DF(\Theta))^{+}\rgrad (\cR \circ F) (\Theta)\\
		&= (DF(\Theta)^\ast DF(\Theta))^{+} P_{T_\Theta \cM}(\grad (\cR \circ F) (\Theta)) \ , \label{eq:linearization}
	\end{align}
	where $P_{T_\Theta \cM}: \cV \to \T_\Theta \cM$ is the orthogonal projection from the ambient space to the tangent space $T_\Theta \cM$ and $\grad$ denotes the Euclidean gradient in $\cV$.
	
	\begin{remark}[Pullback metric]
		Following the result from \cref{cor:ngrad-is-rgrad}, the natural gradient can also be interpreted as a standard gradient with respect to a special Riemannian metric, also called the pullback metric~\cite{lee18}.
		Consider the following symmetric positive semidefinite bilinear form on $T_\Theta \cM$:
		\begin{equation}
			\inner{\zeta}{\xi}_{\Theta} \coloneqq \inner{DF(\Theta)[\zeta]}{DF(\Theta)[\xi]}_{F(\Theta)} \ . \label{eq:induced-riem-metric}
		\end{equation}
		If $DF(\Theta)$ is injective, $\inner{\cdot}{\cdot}_{\Theta}$ is an inner product (Riemannian metric resp.) on $T_\Theta \cM$.
		By definition, $DF(\Theta)$ is an isometry with respect to the new (semi-)inner product.
		Since in this case we have $\ngrad (\cR \circ F)(\Theta) = \rgrad (\cR \circ F) (\Theta)$, the natural gradient can be interpreted as gradient on $\cM$ with respect to the Riemannian metric~\eqref{eq:induced-riem-metric}.
	\end{remark}

	\begin{remark}[Reparametrization]\label{rem:reparametrization}
	In the case that $\cM$ is a manifold and itself accessed through a parametrization $\tilde{F}: \tilde{\cM} \to \cM$ from another Riemannian manifold $\tilde{\cM}$ (or linear space), the natural gradient changes:
	One then has to solve the system
	\begin{equation}
			(D\tilde{F}(\tilde{\Theta})^\ast DF(\tilde{F}(\tilde{\Theta}))^\ast DF(\tilde{F}(\tilde{\Theta}))D\tilde{F}(\tilde{\Theta}))[\chi] = \rgrad (\cR \circ F \circ \tilde{F}) (\tilde{\Theta}) \ ,
	\end{equation}
	where $\tilde{\Theta} \in \tilde{\cM}$ are the new parameters.
	However, by construction, the effective descent direction in $\cH$ remains unchanged.
	\end{remark}

	\subsection{An idealized algorithm}

	Recall that Riemannian gradient descent for minimizing the cost $\cR$ on $\cH$ constructs from a given point $h \in \cH$ an new point $h_\gamma$ according to
	\begin{equation}
		h_\gamma
		= \rR_{h}(- \gamma \rgrad \cR(h))\ , \label{eq:gradient-step-R-h}
	\end{equation}
	where $\gamma > 0$ is an appropriate step size and $\rR_h$ denotes the retraction for $\cH$ at $h$.
	We can interpret this as selecting the new iterate as a point on the curve $\gamma \mapsto h_\gamma$ on $\cH$, which by our assumptions is well defined and smooth for small $\gamma > 0$.
	It means that the function $\mathcal R$ is decreased according to
	\begin{equation}
		 \cR(h_\gamma) = \cR(h) - \gamma \inner{\rgrad \cR(h)}{\rgrad \cR(h)}_{h} + O(\gamma^2) \ . \label{eq:decrease-R-h}
	\end{equation}

	Natural Riemannian descents mimics this update rule, but instead operates on the parameters~$\Theta$ and uses the natural gradient.
	If $h = F(\Theta)$, this results in the update
	\begin{equation}
		\Theta_\gamma
		= \rR_{\Theta}(- \gamma \ngrad (\cR \circ F)(\Theta))  \ . \label{eq:gradient-step-R-theta}
	\end{equation}
	Note that here $\rR_\Theta$ is a retraction on the manifold $\cM$.
 	By construction, the curves $\gamma \mapsto F(\Theta_\gamma)$ and $\gamma \mapsto h_\gamma$ are then equal in first order, as they they both start at $h$ and their derivatives are equal  at $\gamma =0$ due to~\eqref{eq:param-tangent-vector}:
 	\[
		\frac{\mathrm d}{\mathrm d \gamma} F(\Theta_\gamma) \big\rvert_{\gamma = 0} = - DF(\Theta)[\ngrad (\cR \circ F)(\Theta)] = - \rgrad \cR(F(\Theta)) = - \rgrad \cR(h) = \frac{\mathrm d}{\mathrm d \gamma} h_\gamma \big\rvert_{\gamma = 0} \  .
 	\]
	This also implies
	\begin{equation}
		 \cR(F(\Theta_\gamma)) = \cR(h_\gamma) + O(\gamma^2) = \cR(h) - \gamma \inner{\rgrad \cR(h)}{\rgrad \cR(h)}_{h} + O(\gamma^2).
	\end{equation}
	Thus, when starting at $h = F(\Theta)$, then for a single step both update rules \eqref{eq:gradient-step-R-h} and \eqref{eq:gradient-step-R-theta} are equivalent in first order with respect to the step length.
	Note that they are not equal, since the higher-order terms differ.
	Therefore, even if initialized with matching points, both methods will deviate from each other after several iterations.
	(The corresponding continuous gradient flows are identical though.)

	On a similar note, while the natural Riemannian gradient is in theory invariant to the parametrization in every step, updates corresponding to different parametrizations will only be equal in first order, so result in different optimization dynamics in practice, as also pointed out in~\cite{eigel25}.
	More importantly, different parametrization usually even lead to different \emph{empirical} versions of the linear systems which are discussed next. 
	This means that in practice, search directions are generally not equal in first order.

	The pseudocode of the idealized natural Riemannian gradient descent algorithm is given in \cref{algo:ngd}.

	\begin{algorithm}[t]
		\DontPrintSemicolon
		\SetAlgoLined
		\KwIn{Risk $\cR: \cH \to \R$, Parametrization $F: \cM \to \cH$, Initial point $\Theta^{0} \in \cM$}
		$t \leftarrow 1$\;
		\While{not converged}{
			Compute $\rgrad (\cR \circ F)(\Theta^{(t)})$\;
			Solve
			$\left(DF(\Theta^{(t)})^\ast DF(\Theta^{(t)})\right)[\zeta^{(t)}] = \rgrad (R \circ F) (\Theta^{(t)})$
			\;
			Choose step-size $\gamma^{(t)} > 0$\;
			$\Theta^{(t+1)} \leftarrow \rR_{\Theta^{(t)}}\left(- \gamma^{(t)} \cdot \zeta^{(t)}\right)$\;
			$t \leftarrow t + 1$\;
		}
		\Return $\Theta^{(t)}$\;
		\caption{Natural Riemannian Gradient Descent}
		\label{algo:ngd}
	\end{algorithm}

	\subsection{Empirical version}\label{subsec: emprirical version}

	In practice the expectation in the definition of the (true) risk $\cR$ in~\eqref{eq:risk} can be approximated with an empirical expected value of samples $(x^i, y^i)_{k = 1, \ldots, m}$ from the distribution $\QXY$.
	This so-called \emph{empirical risk} is given as
	\begin{equation}\label{eq:empirical-risk}
		\widehat{\cR}_m(F(\Theta)) = \frac{1}{m} \sum_{i=1}^m \ell(F(\Theta), x^i, y^i)
	\end{equation}
	and thus its Riemannian gradient is
	\begin{equation}\label{eq:grad-empirical-risk}
		\rgrad (\widehat{\cR}_m \circ F)(\Theta) = \frac{1}{m} \sum_{i=1}^m \rgrad \ell(F(\cdot), x^i, y^i)(\Theta) \ .
	\end{equation}

	The natural Riemannian gradient of the empirical risk is obtained by replacing the right hand side of~\eqref{eq:ngrad} by~\eqref{eq:grad-empirical-risk}. 
	Its practical computation, however, still requires the operator $DF(\Theta)^\ast DF(\Theta)  : T_\Theta \cM \to T_\Theta \cM$, or at least its action on tangent vectors.
	This can pose a challenge, since there is oftentimes no closed-form expression available for the adjoint of $DF(\Theta)$ with respect to the given Riemannian metric.
	In many relevant applications (such as the ones considered in this work), the hypotheses are functions $h: \cX \to \R^{n_0}$ in $L_2(\cX, \R^{n_0}; \QX)$ where~$\QX$ denotes the marginal probability measure on $\cX$ \wrt $\QXY$, that is, $\QX(X) = \int_{\cX \times \cY} 1_X(x) \QXY(dx, dy)$ where $1_X$ is the indicator function of $X \subseteq \cX$.
	In this setting, computing the operator $DF(\Theta)^\ast DF(\Theta)$ can be achieved through empirical approximation based on the following familiar fact.
	\begin{proposition}
		\label{prop:dfdf-as-integral}
		Assume $\cH = F(\cM) \subseteq L_2(\cX, \R^{n_0}; \QX)$ is an embedded submanifold that is equipped with the Riemannian metric
		$\inner{\eta}{\chi} = \int_\cX \inner{\eta(x)}{\chi(x)} \QX(dx)$.
		For $\Theta \in \cM$ and $x \in \cX$, let
		$DF_x(\Theta): T_\Theta \cM \to \R^{n_0}$ be defined through $DF_x(\Theta)[\xi] = DF(\Theta)[\xi](x)$ for all $\xi \in T_\Theta \cM$.
		Then
		\begin{equation}
			DF(\Theta)^\ast DF(\Theta) 
			= \int_\cX DF_x(\Theta)^\ast DF_x(\Theta)\QX(dx) 
			= \E_{x \sim \QX}[DF_x(\Theta)^\ast DF_x(\Theta)] \ .
		\end{equation}
	\end{proposition}
	\begin{proof}
		For all $\zeta, \xi \in T_\Theta \cM$ we have
		\begin{align}
			\inner{DF(\Theta)[\zeta]}{DF(\Theta)[\xi]}_{F(\Theta)}
			&= \int_\cX \inner{DF(\Theta)[\zeta](x)}{DF(\Theta)[\xi](x)} \QX(dx) \label{eq:l2-inner-product-df}\\
			&= \int_\cX \inner{DF_x(\Theta)[\zeta]}{DF_x(\Theta)[\xi]} \QX(dx)\\
			&= \int_\cX \inner{DF_x(\Theta)^\ast DF_x(\Theta)[\zeta]}{\xi}_\Theta \QX(dx) \\
			&= \inner{\int_\cX DF_x(\Theta)^\ast DF_x(\Theta)[\zeta] \QX(dx)}{\xi}_\Theta\\
			&= \inner{\left(\int_\cX DF_x(\Theta)^\ast DF_x(\Theta)\QX(dx)\right) [\zeta]}{\xi}_\Theta
		\end{align}
		where both the last and the second to last equality follow from the linearity of the integral.
	\end{proof}
	The empirical approximation of $DF(\Theta)^\ast DF(\Theta)$ suggested by \cref{prop:dfdf-as-integral} hence reads
	\begin{equation}
		DF(\Theta)^\ast DF(\Theta) 
		\approx  \frac{1}{m} \sum_{i = 1}^m DF_{x^i}(\Theta)^\ast DF_{x^i}(\Theta) \label{eq:empirical-dfdf} \ ,
	\end{equation}
	where $x_1,\dots,x_m$ are sampled from the distribution $\mu_\cX$.

	Based on~\eqref{eq:grad-empirical-risk} and~\eqref{eq:empirical-dfdf}, the \emph{empirical natural Riemannian gradient} at $\Theta$ is defined	as the solution $\zeta$ to the linear equation
	\begin{equation}
		\frac{1}{m} \sum_{i = 1}^m DF_{x^i}(\Theta)^\ast DF_{x^i}(\Theta) [\zeta]
		 = \frac{1}{m} \sum_{i=1}^m \rgrad \ell(F(\cdot), x^i, y^i)(\Theta) \ . \label{eq:empirical linear equation}
	\end{equation}
	Let us abbreviate the equation with $\widehat{Z}_m(\Theta)[\zeta] = \rgrad (\widehat{\cR}_m \circ F)(\Theta)$. Note that the right hand side is in the image of the symmetric positive semi-definite operator $\widehat{Z}_m(\Theta)$ on $T_\Theta \cM$. Any solution is of the form $\zeta = \widehat{Z}_m(\Theta)^+ \rgrad (\widehat{\cR}_m \circ F)(\Theta) + \eta$, where $\eta$ is in the null space of $\widehat{Z}_m(\Theta)$ and hence orthogonal to $\rgrad (\widehat{\cR}_m \circ F)(\Theta)$. Therefore, the curve $\Theta_\gamma = \rR_{\Theta}(- \gamma \zeta)$ satisfies
	\begin{align}
	 \widehat{\cR}(F(\Theta_\gamma)) &= \widehat{\cR}(F(\Theta)) - \gamma \inner{\rgrad (\widehat{\cR}_m \circ F)(\Theta)}{\zeta} + O(\gamma^2) \\
	 &= \widehat{\cR}(F(\Theta)) - \gamma \inner{\rgrad (\widehat{\cR}_m \circ F)(\Theta)}{\widehat{Z}_m(\Theta)^+ \rgrad (\widehat{\cR}_m \circ F)(\Theta)}_\Theta + O(\gamma^2) \\
	 &\le \widehat{\cR}(F(\Theta)) - \gamma \frac{1}{\| \widehat{Z}_m(\Theta)\|_\Theta } \inner{\rgrad (\widehat{\cR}_m \circ F)(\Theta)}{\rgrad (\widehat{\cR}_m \circ F)(\Theta)}_\Theta + O(\gamma^2) \ ,
	\end{align}
	where $\| \widehat{Z}_m(\Theta)\|_\Theta$ is the spectral norm on $T_\Theta \cM$. This shows, that $-\zeta$ is a gradient related descent direction for the empirical risk. Furthermore, if the number $m$ of samples goes to infinity then $\widehat{Z}_m(\Theta)$ converges to $DF(\Theta)^\ast DF(\Theta)$ and $\rgrad (\widehat{\cR}_m \circ F)(\Theta)$ converges to $\rgrad (\cR_m \circ F)(\Theta)$. The limiting equation hence matches~\eqref{eq:ngrad}. Therefore, for fixed $\Theta$, any accumulation point of a corresponding sequence of solutions $(\zeta_m)$ to~\eqref{eq:empirical linear equation} will be a natural Riemannian gradient for $\cR$.

	The resulting empirical (but still somewhat idealized) version of \cref{algo:ngd} will not be noted separately, as it essentially looks the same except for obtaining $\zeta^{(t+1)}$ from solving~\eqref{eq:empirical linear equation} at $\Theta^{(t)}$.

\section{Natural Gradient for functional tensor network manifolds}\label{sec: Natural Gradient for functional tensor network manifolds}

	We now specialize the above theory to functional low-rank tensor models and develop a practical version of the natural Riemannian gradient algorithm for the learning problems under consideration.

	\subsection{Functional tensor model}\label{sec: functional TN model}

	Let $V_1, \ldots, V_d$ be finite-dimensional vector spaces of real-valued and continuous univariate functions with
	\begin{equation}
		V_\nu = \Span\{ \phi_1^\nu, \ldots, \phi_{n_\nu}^\nu \} \ ,
	\end{equation}
	where $\phi_{j_\nu}^\nu: \Omega_\nu \to \R$ for all $j_\nu= 1, \ldots, n_\nu$ and $\nu = 1,\dots,d$ and $\Omega_\nu \subseteq \R$.
	We assume that $\phi_1^\nu, \ldots, \phi_{n_\nu}^\nu$ is a basis of $V_\nu$ so that $\dim V_\nu = n_\nu$.
	Further, let
	\[
	 \cV = V_1 \otimes \dots \otimes V_d
	\]
	be the tensor product space and let $n_0 \in \N$.

	In the functional tensor model we consider vector-functions $f \in \cV^{n_0}$, that is, functions
	\[
	 f : \Omega_1 \times \dots \times \Omega_d \to \R^{n_0}, \qquad f(x) = \begin{pmatrix} f_1(x) \\ \vdots \\ f_{n_0}(x) \end{pmatrix} \ ,
	\]
	where every component belongs to $\cV$ and hence can be written as
	\begin{equation}
		f_k(x) = \sum_{j_1, \ldots, j_d}^{n_1, \ldots, n_d} \mathsf{A}_{k, j_1, \ldots, j_d} \phi_{j_1}^1(x_1) \phi_{j_2}^2(x_2) \cdots \phi_{j_d}^d(x_d)\ , \qquad k = 1,\dots,n_0 \ . \label{eq:basis representation}
	\end{equation}
	Here $\mathsf{A} \in \R^{n_0 \times n_1 \times \ldots \times n_d}$ is a real-valued $n_0 \times n_1 \times \dots \times n_d$ tensor containing the basis coefficients for all $f_k$ \wrt the tensor product basis $\varphi_{j_1}^1 \otimes \dots \otimes \varphi_{j_d}^d$ of $\cV$.
	Once these basis functions are fixed, the tensor $\mathsf A$ remains the sole parameter for representing functions in $\cV^{n_0}$.

	The basis representation~\eqref{eq:basis representation} can be written in a more compact form. Let
	\begin{equation}
		\Phi_\nu \coloneqq \begin{pmatrix} \phi_1^\nu \\ \vdots \\ \phi_{n_\nu}^\nu \end{pmatrix} \label{eq: basis compact}
	\end{equation}
	for all $\nu = 1, \ldots, d$ and $\Phi: \Omega_1 \times \dots \times \Omega_d \to \R^{n_1 \times \cdots \times n_d}$ be given by
	\begin{equation}
		\Phi(x) \coloneqq \Phi_1(x_1) \otimes \Phi_2(x_2) \otimes \ldots \otimes \Phi_d(x_d) \ , \label{eq:Phi-definition}
	\end{equation}
	which is a rank-$1$ tensor of functions and can be interpreted as a tensor-valued feature map.
	Every $f \in \cV^{n_0}$ may hence be equivalently written as
	\begin{equation} 
		f
		\coloneqq \inner{\mathsf{A}}{\Phi(\cdot)} 
		\coloneqq \inner{\mathsf{A}}{\Phi(\cdot)}_{1, \ldots, d} 
		= \sum_{j_1, \ldots, j_d}^{n_1, \ldots, n_d} \mathsf{A}_{k, j_1, \ldots, j_\nu} \phi_{j_1}^1(\cdot_1) \phi_{j_2}^2(\cdot_2) \cdots \phi_{j_d}^d(\cdot_d) \ ,
	\end{equation}
	where $\inner{\cdot}{\cdot}_{1, \ldots, d}$ denotes the contraction along modes $1, \ldots, d$.
	
	We give a brief example for $\cV^{n_0}$.
	\begin{example}[Polynomial Basis]
		Let $n \in \N$ and $V_1 = V_2 = \ldots = V_d = \R[x]_n$ be the vector spaces of polynomials up to degree $n$.
		For $V_\nu$ we can, for example, choose the monomial basis functions $\phi_j^\nu(\xi) = \xi^{j-1}$ for $j = 1, \ldots, n+1$.
		For $d=2$ and $n = 2$ this results in the feature map
		\begin{equation}
			\vec(\Phi(x))
			= \vec(\begin{pmatrix}1 \\ x_1 \\ x_1^2\end{pmatrix} \otimes  \begin{pmatrix}1 \\ x_2 \\ x_2^2\end{pmatrix})
			= (1, x_1, x_2, x_1^2, x_2^2, x_1 x_2, x_1^2 x_2, x_2^2 x_1, x_1^2 x_2^2)^\T \ .
		\end{equation}
		Note that $\Phi$ is similar, but not identical to the feature map associated to a polynomial kernel.
		For $d=2$ and $n = 2$, the latter is given by
		\begin{equation}
			\Phi_\mathrm{poly}(x) = (1, x_1, x_2, x_1^2, x_2^2, x_1 x_2)^\T
		\end{equation}
		see \eg \cite{shalev-shw14}.
		It generates all multinomial combinations of $x_1$ and $x_2$ of total degree $\leq 2$.
		In contrast, the feature map $\Phi$ also generates higher-order multinomials of maximum degree $\leq 2$.
	\end{example}
	
	\subsection{Functionals tree tensor networks}\label{sec:low-rank-model}

	Since $\mathsf{A} \in \R^{n_0 \times n_1 \times \cdots \times n_d}$ can easily be too large to be stored, we can usually not take the full linear space $\cV^{n_0}$ as a (linear) learning model.
	Instead, further dimensionality reduction is required.
	Functional low-rank tensor models are based on restricting $\cV^{n_0}$ to a low-dimensional set $\cH$ by restricting $\mathsf{A}$ to be an element of some tractable set $\cT \subset \R^{n_0 \times n_1 \times \ldots n_d}$ of low-rank tensors.
	In this way, the class of hypotheses becomes non-linear.
	In this work, for conceptual reasons we consider only the case that $\cT$ is an embedded \emph{manifold} of low-rank tensors.
	A general class of such manifolds are fixed-rank tree tensor networks (TTNs), which includes fixed-rank versions of the the well-known Tucker format, the tensor train (TT) format~\cite{oseledets11}, or the Hierarchical Tucker (HT) format~\cite{HackbuschKuehn09} as special cases.
	The TT format will be briefly explained further below.
	Manifold properties for the fixed-rank versions of these examples have been worked out in~\cite{Holtzetal12,UschmajewVandereycken13} and are by now well understood.
	The limitation to fixed rank can be addressed by mixing rank-adaptive strategies with fixed-rank schemes.
	In the following, for any such manifold $\cT$, let
	\begin{equation}
		G: \cT \to \cV^{n_0} \ , \quad G(\mathsf{A}) \coloneqq \inner{\mathsf{A}}{\Phi(\cdot)} \ , \label{eq:map G}
	\end{equation}
	and define
	\begin{equation}
	\cH \coloneqq \Im G.
	\end{equation}
	Note that $\cH$ is a manifold since $G$ is a linear isomorphism.
	
	Natural Riemannian gradient descent \wrt the parametrization $G$ could in theory directly be applied on the embedded manifolds $\cT \subseteq \R^{n_0 \times n_1 \times \dots \times n_d}$ using the concepts of Riemannian optimization.
	Specifically, it requires solving~\cref{eq:ngrad} on tangent spaces.
	This is certainly feasible for tree tensor networks.
	For example, for fixed-rank TT manifolds the required machinery has been applied in several works, e.g., in \cite{kressner16, steinlechner16, Novikov18, Rakhuba19, UschmajewVandereycken20}.
	Notably, in~\cite{eigel25}, natural gradient descent has been applied on a TT manifold.
	In this work, we will follow a more direct approach which accounts for the fact that low-rank tensors are usually more conveniently accessed via another parametrization map
	\begin{equation}
		\tau: \cM \to \cT \ , \qquad \Theta \mapsto \tau(\Theta) = \tau(\Theta_1,\dots,\Theta_{d'}) = \mathsf{A} \ , \label{eq:map tau}
	\end{equation}
	where $\cM = U_1 \times U_2 \times \cdots \times U_{d'}$ for some embedded manifolds $U_k \subseteq \R^{m_k}$, $k=1,\dots,d'$.
	The map $\tau$ is usually multilinear in the sense that there is a multilinear map $\tilde{\tau}: \R^{m_1} \times \cdots \times \R^{m_{d'}} \to \mathrm{cl}(\cT)$ such that $\tau = \tilde{\tau}|_\cT$, where $\mathrm{cl}(\cT)$ denotes set closure.
	In particular, the differential $D\tau(\Theta)$ is given through the Leibniz product rule:
	\begin{equation}
	 D \tau(\Theta)[\delta \Theta_1,\dots,\delta \Theta_{d'}] = \tau(\delta \Theta_1,\Theta_2,\dots,\Theta_{d'}) + \dots + \tau(\Theta_1,\dots,\Theta_{d'-1},\delta \Theta_{d'}).
	\end{equation}
	This is itself a sum of low-rank tensors and can be efficiently evaluated in practical implementations in the relevant examples.

	Overall, the parametrization we then use is
	\begin{equation}\label{eq: FTN parametrization}
	 F : \cM \to \cH, \qquad F(\Theta) = (G \circ \tau)(\Theta) = \inner{\tau(\Theta)}{\Phi(\cdot)} \ .
	\end{equation}
	For concreteness, we briefly discuss the manifold $\cM$ and the map $\tau$ for the fixed-rank TT manifold and for balanced binary TTNs.

	\begin{example}[TT format]\label{ex:TT format}
	In the functional TT format, we consider functions $\inner{\mathsf{A}}{\Phi(\cdot)} \in \cV^{n_0}$ where the coefficient tensor $\mathsf A \in \R^{n_0 \times \cdots \times n_d}$ admits an entry-wise decomposition
	\begin{equation}
		 \mathsf{A}_{j_0, j_1, \ldots, j_d}
		= \sum_{k_1, \ldots, k_d}^{r_1, \ldots, r_d}
		A_0(j_0, k_1) A_1(k_1, j_1, k_2) A_2(k_2, j_2, k_3) \cdots A_{d-1}(k_{d-1}, j_{d-1}, k_d) A_d(k_d, j_d) \  \label{eq:TT}
	\end{equation}
	for some $A_k \in \R^{r_k \times n_k \times r_{k+1}}$, $k=0,\dots,d$, where $r_0 = r_{d+1} = 1$.\footnote{Note that in this setting the mode sizes of TT cores would commonly be enumerated as $r_{-1}, \ldots, r_{d}$. In order to avoid negative indices we adopt a shifted enumeration starting at $0$.}
	The vector $\mathbf{r} = (r_1, \ldots, r_d)$ is called the TT-rank of $\mathsf{A}$, assuming all values $r_k$ are as small as possible for such a representation to exist.
	We can then conversely consider the manifold of all tensors with a fixed TT-rank $\mathbf{r}$. 
	In this case we choose
	\begin{equation}
		\cM =  \R_\ast^{r_0 \times n_0 \times r_{1}} \times \cdots \times \R_\ast^{r_d \times n_d \times r_{d+1}} \label{eq:M-TT} \ ,
	\end{equation}
	where by  $\R^{r_k \times n_k \times r_{k+1}}_\ast$ we denote the open set of third-order tensors whose $(r_k \times n_k r_{k+1})$ and $(r_k n_k \times r_{k+1})$ unfolding matrices have full row resp. column rank (equal to $r_{k}$ resp.~$r_{k+1}$).
	The map $\tau: \cM \to \cT$, $(A_0, A_1, \ldots, A_d) \mapsto \tau(A_0,A_1, \ldots, A_d) = \mathsf{A}$ is then implicitly given by \cref{eq:TT}. We refer to, \eg,~\cite[Section~9.3]{UschmajewVandereycken20} for further details.
	Due to an inherent non-uniqueness in the representation~\eqref{eq:TT} it is in theory possible to further restrict all but one set $\R^{r_k \times n_k \times r_{k+1}}_\ast$ to certain Stiefel manifolds.
	However, we limit the discussion of non-uniqueness to the next example of balanced binary TTNs since we do not use TT in our numerical experiments.
	We also remark that the core with the output mode of size $n_0$ could in theory be put in any position in the TT.
	We move it to the first position purely because of notational simplicity, not because it is in any way canonical.
	\cref{fig:ttn} (left) depicts a functional tensor train in tensor diagram notation.
	The cores in the top row are the parameters $A_0, A_1, \ldots, A_d$.
	The tensors in the bottom row are the (evaluated) features vector $\Phi_\nu(x_\nu)$.
	\end{example}

	\begin{figure}[t]
		\centering
		\hspace{-6mm}\begin{minipage}{0.45\linewidth}
			\centering
			\begin{tikzpicture}[>=latex, node distance=15mm]
				\tikzset{ every node/.append style={font=\scriptsize} }
				
				\tikzstyle{tensor}=[circle, fill=black, minimum size=2mm, inner sep=0pt]
				
				\node[tensor] (G1) {};
				\node[tensor] (G2) [right of=G1] {};
				\node[tensor] (G3) [right of=G2] {};
				\node[tensor] (G4) [right of=G3, xshift=5mm] {};
				\node at ($(G3)!0.5!(G4)$) {$\cdots$};
				\node[tensor] (G5) [below of=G2, yshift=5mm] {};
				\node[tensor] (G6) [below of=G3, yshift=5mm] {};
				\node[tensor] (G7) [below of=G4, yshift=5mm] {};
				\node at (G5.south) [right=-2pt, below=1pt] {$\Phi^1(x_1)$};
				\node at (G6.south) [right=-2pt, below=1pt] {$\Phi^2(x_2)$};
				\node at (G7.south) [right=-2pt, below=1pt] {$\Phi^d(x_d)$};
				
				\draw (G1) -- (G2);
				\draw (G2) -- (G3);
				\draw (G2) -- (G5) node[above=5pt, xshift=-6pt] {$n_1$};
				\draw (G3) -- (G6) node[above=5pt, xshift=-6pt] {$n_2$};
				\draw (G4) -- (G7) node[above=5pt, xshift=-6pt] {$n_d$};
				\draw (G3) -- ++(0.5, 0);
				\draw (G4) -- ++(-0.5, 0);

				\node[above] at ($(G1)!0.5!(G2)$) {$r_1$};
				\node[above] at ($(G2)!0.5!(G3)$) {$r_2$};
				
				\draw (G1) -- ++(0, -10mm) node[above=5pt, xshift=-6pt] {$n_0$};
				
				\node at (G1.north) [above = -3pt] {$A_0$};
				\node at (G2.north) [above = -3pt] {$A_1$};
				\node at (G3.north) [above = -3pt] {$A_2$};
				\node at (G4.north) [above = -3pt] {$A_d$};
				
			\end{tikzpicture}
		\end{minipage}
		\hspace{8mm}
		\begin{minipage}{0.45\linewidth}
			\centering
			\begin{tikzpicture}[>=latex, node distance=14mm]
				\tikzset{ every node/.append style={font=\scriptsize} }
				\tikzstyle{tensor}=[circle, fill=black, minimum size=2mm, inner sep=0pt]
				
				\node[tensor] (A1) at (0,0) {};
				\node[tensor] (A2) [right of=A1, xshift=-3mm] {};
				\node[tensor] (A3) [right of=A2, xshift=-4mm] {};
				\node[tensor] (A4) [right of=A3, xshift=-3mm] {};
				
				\node (A5) [right of=A4, yshift=3mm, xshift=-7mm] {$\cdots$};
				
				\node[tensor] (A6) [right of=A4, xshift=2mm] {};
				\node[tensor] (A7) [right of=A6, xshift=-4mm] {};
				
				\node[tensor] (B1) at ($(A1)!0.5!(A2)+(0,0.8)$) {};
				\node[tensor] (B2) at ($(A3)!0.5!(A4)+(0,0.8)$) {};
				\node[tensor] (B3) at ($(A6)!0.5!(A7)+(0,0.8)$) {};
				
				\node[tensor] (F1) at ($(B1)!0.5!(B2)+(0,0.8)$) {};
				
				\node[tensor] (C1) at ($(A1)!0.5!(A7)+(0,3)$) {};

				\draw (A1) -- (B1) node[midway, left] {$n_1$};
				\draw (A2) -- (B1) node[midway, right] {$n_2$};
				\draw (A3) -- (B2) node[midway, left] {$n_3$};
				\draw (A4) -- (B2) node[midway, right] {$n_4$};
				
				\draw (A6) -- (B3) node[midway, left] {$n_{d-1}$};
				\draw (A7) -- (B3) node[midway, right] {$n_{d}$};
				
				\draw (B1) -- (F1) node[midway, left] {$r_{1, 2}$};
				\draw (B2) -- (F1) node[midway, right] {$r_{3, 4}$};
				\node at ($(F1)!0.5!(C1)+(-0.1,0)$) {\rotatebox{89.5}{$\ddots$}};
				\node at ($(A4)!0.5!(A6)+(-0.1, 2.2)$) {\rotatebox{-0.5}{$\ddots$}};

				\draw (C1) -- ++(0, 0.5) node[above] {$n_0$};
				
				\node at (A1.south) [xshift=-2pt, below=-1pt] {$\Phi^1(x_1)$};
				\node at (A2.south) [xshift=-2pt, below=-1pt] {$\Phi^2(x_2)$};
				\node at (A3.south) [xshift=-2pt, below=-1pt] {$\Phi^{3}(x_3)$};	
				\node at (A4.south) [xshift=-2pt, below=-1pt] {$\Phi^4(x_4)$};	
				\node at (A6.south) [xshift=-2pt, below=-1pt] {$\Phi^{d-1}(x_{d-1})$};	
				\node at (A7.south) [xshift=8pt, below=-1pt] {$\Phi^d(x_d)$};	
				
				\node at (B1.north) [left] {$A_{1, 2}$};
				\node at (B2.north) [right] {$A_{3, 4}$};
				\node at (B3.north) [right] {$A_{d-1, d}$};
				\node at (F1.north) [left] {$A_{1, 2, 3, 4}$};
				\node at (C1.north) [left] {$A_{1, \ldots, d}$};
				
			\end{tikzpicture}
		\end{minipage}
		
		\caption{Functional tensor train (left) and functional (binary) tree tensor network (right).}
		\label{fig:ttn}
	\end{figure}
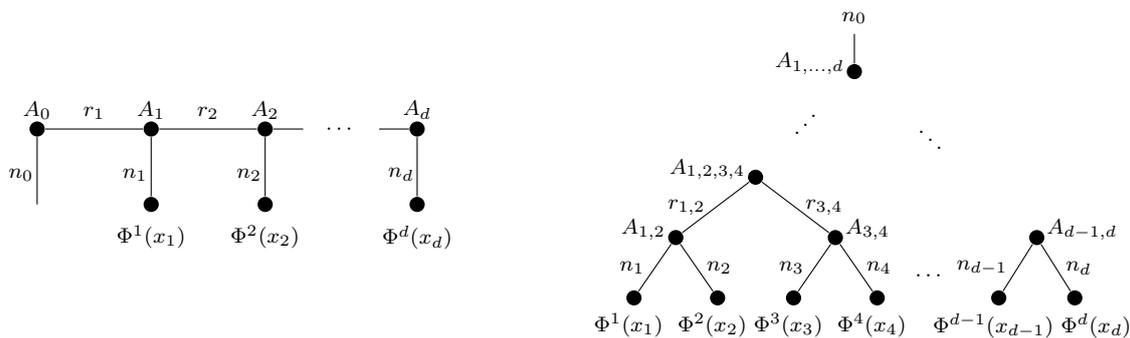
	
	\begin{example}[Balanced binary TTN]\label{ex: balanced binary TTN}
	Functional tree tensor networks (TTNs) are a generalization of functional TTs.
	We provide an example of a binary balanced functional TTN in \cref{fig:ttn} (right).
	The generalization to non-balanced binary trees is straightforward.
	For a formal definition of the depicted TTN, let $d = 2^k$.
	We label all the nodes of the balanced tree with $t \in 2^{\{1, \ldots, d\}}$, $t = \{(i-1) 2^j + 1, \ldots, i 2^j\}$, $i = 1, \ldots, 2^{k-j}$, $j = 1, \ldots, k$.
	These index sets $t$ form nested partitions of $\{1, \ldots, d\}$. Specifically, $t_L$ and $t_R$ are the left and right children of $t$ iff $t = t_L \cup t_R$ and $t_L < t_R$ elementwise.
	To each node $t$ that is not a leaf and has left and right children $t_L$ and $t_R$, respectively, we then attach third-order tensors $A_t \in \R^{r_{t_L} \times r_{t_R} \times r_t}$, called the \emph{cores} of the TTN.
	This requires to fix integers $r_t$ for every node $t$ that determine the sizes of cores.
	For leaves $t=\{\nu\}$ we enforce $r_t = n_{\nu}$, whereas for the root $t = \{1,\dots,d\}$ we should take $r_t = n_0$.
	The TTN then implicitly associates another set of matrices $B_t$ to all nodes $t$ that are not leaves according to the following recursive construction:
	\begin{align} \label{eq:ttn_recursion}
		B_{t} =
		\begin{cases}
			\bar A_{t}, &\text{if $\abs{t} = 2$ (i.e.~$t= \{\nu, \nu+1\}$)}, \\
			(B_{t_L} \otimes B_{t_R}) \bar A_t, &\text{if $\abs{t} \ge 4$ (i.e.~$t = t_L \cup t_R$)} \, .
		\end{cases}
	\end{align}
	Here $\bar A_t$ is a reshape of the core $A_t$ into an $(r_{t_L}r_{t_R}) \times r_t$ matrix.
	As a result, each matrix $B_t$ is of size $n_t \times r_t$, where $n_t$ is the product of dimensions $n_\nu$ over all $\nu \in t$, e.g.~$n_{\{1,2\}} = n_1n_2$, $n_{\{1,2,3,4\}} = n_1n_2n_3n_4$ etc.
	In particular, the matrix $B_{\{1,\dots,d\}} \in \R^{(n_1\cdots n_d) \times n_0}$ at the root encodes a tensor $\mathsf{A} \in \R^{n_0 \times n_1 \times \dots \times n_d}$.
	The recursive relation~\eqref{eq:ttn_recursion} induces a multilinear map
	\begin{equation}
		\mathsf{A} = \tau( (A_t) ) \ .
	\end{equation}
	acting on the set $\bigtimes_{\abs{t} \ge 2} \R^{(r_{t_L} r_{t_R}) \times r_t}$ of tuples $(A_t)$ of all cores of specified size.
	The image of $\tau$ is the set $\cT$ of all tensors $\mathsf{A}$ representable in the described recursive way for the fixed choice of bond dimensions $(r_t)$.
	The corresponding functional TTN model $\cH$ is obtained via contractions $\inner{\mathsf{A}}{\Phi(x)}$ with rank-one tensors $\Phi(x)$.
	The key point is that performing these contractions is practically feasible using recursion, if the inner sizes $r_t$, also called \emph{bond dimensions} of the TTN, are moderate, since all computations are performed using only the cores $A_t$.
	The matrices $B_t$ are never explicitly formed.
	Note that the described format slightly varies from the original HT format from~\cite{HackbuschKuehn09} in that no separate cores (bases) were attached to the leaves.
	However, it corresponds to the balanced HT format for (the slices) of a $n_0 \times (n_1n_2) \times \dots \times (n_{d-1}n_d)$ reshape of $\mathsf{A}$.
	Note that in our example the output mode of dimension $n_0$ is attached to the root of the tree.
	It could, in theory also be attached to any other core, however, the root seems to be the most canonical choice.
	Now, in order to obtain a smooth manifold, as formally required for the purpose of this work, we would need to further restrict the cores tot satisfy $\rank(\bar A_t) = r_t$ for all nodes except for the root.
	This is possible if and only if $r_t \le r_{t_L} r_{t_R}$ for all $t$ with $2 \le \abs{t} < d$ and implies that all matrices $B_t$ have full column rank $r_t$.
	Without going into detail we note that then all reshaped cores $\bar A_t$ except at the root can be further restricted to Stiefel manifolds without changing the image of $\tau$.
	We refer to~\cite{UschmajewVandereycken13} or~\cite[Section~3.8]{Bachmayr16} for details.
	In summary, for the reshaped cores~$\bar A_t$ the parameter space $\cM$ could be taken as a Cartesian product of Stiefel manifolds except for the root:
	\begin{equation}
		\cM = \R^{(r_{\{1,\dots,d/2\}} r_{\{d/2 + 1,\dots,d\}}) \times n_0} \times \left(\bigtimes_{2 \le \abs{t} < d } \operatorname{St}(r_{t_L}r_{t_R}, r_t) \right) \label{eq:M-TTN}
	\end{equation}
	While this restriction to Stiefel manifolds still does not eliminate all non-uniqueness of the parametrization, it already improves the numerical stability of the TTN representation.
	For addressing the remaining non-uniqueness a quotient formalism can be employed, as discussed in the following remark.
	\end{example}
	
	\begin{remark}[Quotient formalism] \label{rem: quotient formalism}
	The explicit multilinear parametrizations $\tau : \cM \to \cT$ of low-rank tensor manifolds as in~\eqref{eq:map tau} are convenient and surjective but typically not injective.
	Correspondingly, the parametrization $F = G \circ \tau$ of the functional low-rank model $\cH$ will not be injective, where $G$ is the map~\eqref{eq:map G}.
	For example, the representation of balanced binary TTNs vie Stiefel manifolds as in \cref{eq:M-TTN} still includes orthogonal invariances between the inner nodes of the tree.
	As discussed in \cref{sec: natural Riemannian gradient} the natural Riemannian gradient will then generally not be unique.
	This motivates employing a quotient manifold formalism on $\cM$ for uniquely representing tangent vectors of $\cT$.
	For a detailed description of the quotient formalism for functional TTNs, we refer to~\cite{SilvaHerrmann15} and \cite{willner25}.
	Here we only state what is relevant to our context.
	Concretely, we employ a Cartesian horizontal space denoted by $H_\Theta^\equiv \cM$,
	and the orthogonal projector onto the Cartesian horizontal space, $P_\Theta^\equiv:  T_\Theta\cM \to H_\Theta^\equiv \cM$.
	The restriction of $D\tau$ to $H_\Theta^\equiv \cM$ then is bijective.
	This means that the restriction of $F$ to $H_\Theta^\equiv \cM$ is a local isomorphism and thus there is a unique solution to~\cref{eq:ngrad} in $H_\Theta^\equiv \cM$.
	When restricting to $H_\Theta^\equiv \cM$, the system for the natural Riemannian gradient can also be written as
	\begin{equation}
		P_\Theta^\equiv DF(\Theta)^\ast DF(\Theta) P_\Theta^\equiv [\zeta] = P_\Theta^\equiv \rgrad (\cR \circ F)(\Theta) \label{eq:p-equiv-ngrad-system}
	\end{equation}
	and needs to be solved for $\zeta \in H_\Theta^\equiv \cM$. By applying $DF(\Theta)$, it can then be easily verified that the solution to this system is a natural Riemannian gradient:
	\begin{align}
		DF(\Theta) [\zeta]
		&=  DF(\Theta) [P_\Theta^\equiv \zeta]\\
		&= DF(\Theta) [P_\Theta^\equiv ((DF(\Theta)P_\Theta^\equiv)^\ast (DF(\Theta)P_\Theta^\equiv))^+ (DF(\Theta)P_\Theta^\equiv)^\ast [\rgrad \cR(F(\Theta))]]\\
		&= (DF(\Theta) P_\Theta^\equiv) (DF(\Theta)P_\Theta^\equiv)^+ [\rgrad \cR(F(\Theta))] = \rgrad \cR(F(\Theta)).
	\end{align}
	Notably, formula~\eqref{eq:p-equiv-ngrad-system} coincides with \cite[Equation 3.2]{hu24}. 
	Since their framework also extends to optimization on quotient manifolds,
	many of their information-geometric considerations and findings apply.
	\end{remark}

	We conclude this subsection with a discussion on the choice of bases for the spaces $V_\nu$.
	The functional tensor model assumed a fixed vector of basis functions $\Phi_\nu = (\varphi_1^\nu,\dots,\varphi_{n_\nu}^\nu)^\T$ for each $V_\nu$, which enter through the feature map $\Phi$ in the map $G$ in~\eqref{eq:map G} and hence in the overall parametrization $F = G \circ \tau$.
	Let us indicate this by writing $F_\Phi$ instead of $F$.
	While the space $\cV^{n_0}$ containing all functions of the form $\inner{\mathsf{A}}{\Phi(\cdot)}$ obviously does not depend on the particular choice of basis, one may ask how it influences the restriction to a low-rank TTN manifold.
	Assume the bases are changed according to
	\[
	 \Psi^\nu = M_\nu^{-1} \Phi_\nu
	\]
	for $\nu = 1,\dots,d$ where $M_\nu \in \R^{n_\nu \times n_\nu}$ are invertible.
	Let $\Psi(\cdot) = \Psi^1(\cdot_1) \otimes \dots \otimes \Psi^d(\cdot_d)$ be the corresponding feature map.
	Then for any $x \in \Omega_1 \times \dots \times \Omega_d$ we have
	\[
	 \Phi(x) = \Phi_1(x_1) \otimes \dots \otimes \Phi_d(x_d) = (M_1 \Psi^1(x_1) \otimes \dots \otimes (M_d \Psi^d(x_d)) = (M_1 \otimes \dots \otimes M_d) \Psi(x)
	\]
	where $M_1 \otimes \dots \otimes M_d$ denotes the tensor product of linear maps.
	Correspondingly, for a functional TTN $F_{\Phi}(\Theta)$ we have
	\[
	 F_\Phi(\Theta) = \inner{\tau(\Theta)}{\Phi(\cdot)} = \inner{(I_{n_0} \otimes M_1^\T \otimes \dots \otimes M_d^\T) \tau(\Theta)}{\Psi(\cdot)}.
	\]
	The considered TTN manifolds $\cT \in \R^{n_0 \times n_1 \times \dots \dots \times n_d}$ are invariant under tensor product of invertible linear maps, so $\tau(\Theta) \in \cT$ implies $(I_{n_0} \otimes M_1^\T \otimes \dots \otimes M_d^\T) \tau(\Theta) \in \cT$.
	Therefore
	\[
	 (I_{n_0} \otimes M_1^\T \otimes \dots \otimes  M_d^\T) \tau(\Theta)  = \tau(\tilde \Theta)
	\]
	for some $\tilde \Theta \in \cM$. 
	Finding $\tilde \Theta$ only requires applying the change of basis to the cores connected to the leaves of the TTN. 
	For example, for the TT format from \cref{ex:TT format} one easily verifies
	\[
	 (I_{n_0} \otimes M_1^\T \otimes \dots \otimes M_d^\T) \tau(A_0,A_1,\dots,A_d) = \tau(\tilde A_0,\tilde A_1,\dots, \tilde A_d)
	\]
	with $\tilde A_0 = A_0$ and $\tilde A_\nu(k_\nu,\cdot,k_{\nu+1}) = M_\nu A_\nu(k_\nu,\cdot,k_{\nu+1})$ for $\nu=1,\dots,d$. 
	As a result we obtain
	\[
		F_{\Phi}(\Theta) = F_{\Psi}(\tilde \Theta).
	\]
	In conclusion, the change of the tensor product basis can be simply interpreted as a (linear!) reparametrization of the same manifold similar as discussed in \cref{rem:reparametrization}.
	By construction, the corresponding natural Riemannian gradient will automatically adjust and revert the effect, leading to the same descent direction on $\cH$ (at corresponding parameters).
	For practical purposes, however, the choice of basis may still be relevant as can be observed in experiments; see \cref{sec:Least-squares recovery problem}.
	It was already mentioned that reparametrizations lead to equivalent natural Riemannian gradients only in first order, so methods will deviate over many iterations.
	Another aspect is that the conditioning of $DF^* DF$ and its empirical counterpart can be affected by the choice of basis.
	This is also indicated by the discussion in~\cref{remark: meaure product structure} on orthonormal bases.
	
	\subsection{Least-squares regression with functional tree tensor networks}
	\label{sec:least_squares_ftn}
	We now discuss how to employ the functional low-rank model for least-squares regression, which was also considered in~\cite{eigel25}.
	We set $\cY = \R^{n_0}$ and consider hypotheses $f \in L_2(\cX, \R^{n_0}; \QX)$ together with the least-squares loss
	\begin{equation}
		\ell(f, x, y) \coloneqq \norm{f(x) - y}_\cY^2 \ ,
	\end{equation}
	where $\norm{\cdot}_\cY$ is any norm induced by an inner product on $\cY$. 	As discussed in the introduction, we assume that for each $x \in \cX$ there is a unique $y = y(x)$ such that $\E_{(x, y) \sim \QXY}[\ell(f, x ,y)] = \E_{x \sim \QX}[\ell(f, x , y(x))]$.
	The risk then becomes
	\begin{equation}
		\cR(f) = \int_\cX \norm{f(x) - y(x)}_\cY^2 \QX(dx) = \| f - y \|^2_{L_2(\cX, \cY; \QX)}\ . \label{eq:risk least squares}
	\end{equation}

	As a learning model we choose a functional low-rank manifold $\cH \subseteq L_2(\cX, \R^{n_0}; \QX)$.
	Concretely, we choose $\mathcal H = \Im G$ as in the previous section and consider the parametrization via $F : \cM \to \cH$, $F(\Theta) = (G \circ \tau)(\Theta) = \inner{\tau(\Theta)}{\Phi(\cdot)}$ as in~\eqref{eq: FTN parametrization}.
	Note that assuming $\cH \subseteq L_2(\cX, \R^{n_0}; \QX)$ in this context may require some additional conditions, but does not seem critical.
	For example if the marginal measure $\QX$ has an integrable density on $\cX$, assuming all basis functions $\varphi_i^\nu$ to be continuous and square integrable is sufficient.
	In particular, $\cX = \Omega_1 \times \dots \times \Omega_d \subseteq \R^d$.

	We will equip $\cH$ with the ``trivial'' Riemannian metric
	\begin{equation}
		\inner{\zeta}{\xi}_f \coloneqq \int_\cX \inner{\zeta(x)}{\xi(x)}_\cY \QX(dx) \ , \label{eq:ls-metric}
	\end{equation}
	where $\zeta, \xi \in T_f \cH$.
	In particular, the metric is the same at any point.
	Of course, there are other choices for the Riemannian metric on $\cH$, but we do not consider this case for the least-squares setting.
	Note that for the specific choice~\eqref{eq:ls-metric}, the natural Riemannian gradient descent method matches the Gauss-Newton method for minimizing the risk~\eqref{eq:risk least squares} in parameterized form $f = F(\Theta)$.

	Let us now consider how to compute $DF(\Theta)^\ast DF(\Theta)$ by first investigating $DG(\mathsf{A})^\ast DG(\mathsf{A})$ for a tensor $\mathsf{A}$. Thanks to our choice of the metric~\eqref{eq:ls-metric},~\cref{prop:dfdf-as-integral} is applicable.
	\begin{proposition}
		It holds that 
		\begin{equation}
			DG(\mathsf{A})^\ast DG(\mathsf{A}) =  I_{n_0} \otimes \int_\cX \Phi_1(x_1) \Phi_1(x_1)^\T \otimes \ldots \otimes \Phi_d(x_d) \Phi_d(x_d)^\T \QX(dx) \ .
		\end{equation}
		where $\otimes$ denotes the tensor product of linear operators.
	\end{proposition}
	\begin{proof}
		By~\cref{prop:dfdf-as-integral}, we have
		\begin{equation}
			DG(\mathsf{A})^\ast DG(\mathsf{A}) = \int_\cX DG_x(\mathsf{A})^\ast DG_x(\mathsf{A})\QX(dx) \ .\label{eq: integration DGx}
		\end{equation}
		Since $G_x(\mathsf{A}) = \inner{\mathsf{A}}{\Phi(x)}$ is linear in $\mathsf{A}$, we have $D G_x(\mathsf{A}) = \inner{\cdot}{\Phi(x)}$.
		For rank-one tensors $\mathsf{B} = b_0 \otimes b_1 \otimes \dots \otimes b_d$ and $\mathsf{C} = c_0 \otimes c_1 \otimes \dots \otimes c_d$ one directly verifies
		\begin{align}
		 \inner{\mathsf{C}}{D G_x(\mathsf{A})^\ast D G_x(\mathsf{A}) (\mathsf{B})}_{0,1, \ldots, d}
		 &= \inner{D G_x(\mathsf{A})(\mathsf{C})}{D G_x(\mathsf{A}) (\mathsf{B})}_0\\
 		 &= \inner{ b_0 \cdot [b_1^\T \Phi_1(x_1) \cdots b_d^\top \Phi_d(x_d)] }{ c_0 \cdot [c_1^\T \Phi_1(x_1) \cdots c_d^\top \Phi_d(x_d)] }_0\\
 		 &= b_0^\T c_0 \cdot  b_1^\T \Phi_1(x_1) \Phi_1(x_1)^\T c_1  \cdots b_d^\T \Phi_d(x_d) \Phi_d(x_d)^\T c_d \\
 		 &= \inner{\mathsf{B}}{[I_{n_0} \otimes \Phi_1(x_1) \Phi_1(x_1)^\T \otimes \ldots \otimes \Phi_d(x_d) \Phi_d(x_d)^\T]\mathsf{C}}_{0,1,\dots,d}.
		\end{align}
		The formula then extends to all tensors $\mathsf{B}, \mathsf{C} \in \R^{n_0 \times n_1 \times \dots \times n_d}$ and the result follows by~\eqref{eq: integration DGx}.
	\end{proof}

	From the lemma we obtain
	\begin{align}
		DF(\Theta)^\ast DF(\Theta) 
		&= D\tau(\Theta)^\ast  \left(I_{n_0} \otimes \int_\cX \Phi_1(x_1) \Phi_1(x_1)^\T \otimes \ldots \otimes \Phi_d(x_d) \Phi_d(x_d)^\T \QX(dx) \right) D\tau(\Theta) \\
		&= \int_\cX D\tau(\Theta)^\ast  \left(I_{n_0} \otimes \Phi_1(x_1) \Phi_1(x_1)^\T \otimes \ldots \otimes \Phi_d(x_d) \Phi_d(x_d)^\T \right) D\tau(\Theta) \QX(dx) \ .
	\end{align}
	The integral cannot be computed exactly and is approximated empirically according to \cref{eq:empirical-dfdf}, which results in
	\begin{equation} \label{eq: empirical system l2}
		DF(\Theta)^\ast DF(\Theta) 
		\approx \frac{1}{m} \sum_{i = 1}^m  D\tau(\Theta)^\ast \left(I_{n_0} \otimes \Phi_1(x^i_1) \Phi_1(x^i_1)^\T \otimes \cdots \otimes \Phi_d(x^i_d) \Phi_d(x^i_d)^\T \right) D\tau(\Theta) \ .
	\end{equation}
	The empirical linear system for the natural Riemannian gradient given by~\eqref{eq:empirical linear equation} then becomes
	\begin{multline}
		\frac{1}{m} \sum_{i = 1}^m  D\tau(\Theta)^\ast \left(I_{n_0} \otimes \Phi_1(x^i_1) \Phi_1(x^i_1)^\T \otimes \cdots \otimes \Phi_d(x^i_d) \Phi_d(x^i_d)^\T \right) D\tau(\Theta) [\zeta]
		\\ {}= \frac{1}{m} \sum_{i=1}^m \rgrad \ell(F(\cdot), x^i, y^i)(\Theta) \ .
	\end{multline}

	\begin{remark}[$\QX$ with product structure]\label{remark: meaure product structure}
		Computing the Riemannian gradient becomes considerably easier if the probability measure $\QX$ is a product measure, that is, if $\QX$ factorizes into
		\begin{equation}
			\QX(X) = Q_1(X_1) Q_2(X_2)  \cdots Q_d(X_d)
		\end{equation}
		for $X = X_1 \times X_2 \times \cdots \times X_d$.
		This can be the case in applications where it is possible to choose the type of sampling, such as problems involving PDEs.
		In this case it is possible to choose a tensor product basis $\{ \psi_{i_1}^1 \otimes \cdots \otimes \psi_{i_d}^\nu\}$ of $\cV$ that is orthonormal in $L_2(\cX;\QX)$ by choosing each basis $\{\psi_{1}^\nu \otimes \cdots \otimes \psi_{n_\nu}^\nu \}$ orthonormal in $L_2(\Omega_\nu;Q_\nu)$.
		Define $\Psi: \R^d \to \R^{n_1 \times \cdots \times n_d}$ similar to \cref{eq:Phi-definition} and set 
		\begin{equation}
			\tilde{G}(\Theta) = \inner{\Theta}{\Psi(\cdot)} \ .
		\end{equation}
		It is easy to verify that $\Im \tilde{G} = \Im G$, as already discussed in \cref{sec:low-rank-model}.
		By construction, $\tilde{G}$ is an isometry.
		Hence $\ngrad (R \circ \tilde{G}) = \rgrad (R \circ \tilde{G})$ and for $\tilde{F} \coloneqq \tilde{G} \circ \tau$ we have
		\begin{equation}
			D\tilde{F}(\Theta)^\ast D\tilde{F}(\Theta) 
			= D \tau(\Theta)^\ast \underbrace{D \tilde{G}(\tau(\Theta))^\ast D \tilde{G}(\tau(\Theta))}_{= I} D\tau(\Theta)
			= D \tau(\Theta)^\ast D\tau(\Theta) \ .
		\end{equation}
		The operator $D \tau(\Theta)^\ast D\tau(\Theta)$ can even be evaluated exactly and does not have to be approximated empirically.
		This is exactly what \citet{SilvaHerrmann15} employ for their approximate Gauss-Newton method.
		Using $D \tau(\Theta)^\ast D\tau(\Theta)$ as a substitute for $D F(\Theta)^\ast DF(\Theta)$ when arbitrary measures and bases are involved of course completely discards functional considerations and can then really be seen as undoing the effect of the reparametrization~$\tau$.
		Indeed, applying $D \tau(\Theta) (D \tau(\Theta)^\ast D\tau(\Theta))^{+}$ to $\rgrad(R \circ F)$ recovers $\rgrad (R \circ G)$, independently of the basis used.
	\end{remark}

	\subsection{Multinomial logistic regression with functional tree tensor networks}
	\label{sec:logistic-regression}

	For classification tasks, we consider multinomial logistic regression, also called softmax regression.
	Here, the task is to classify inputs $x \in \cX$ into $n_0$ classes.
	Let $S_{n_0} \coloneqq \{p \in \R^{n_0} \mid \sum_j p_j = 1, p_j > 0\}$ denote the open (probability) simplex.
	Recall that in this setting, for hypotheses $f: \cX \to S_{n_0}$, the risk is generally given by
	\begin{equation}
		\cR(f) = - \int_\cX \sum_{j = 1}^{n_0} y(x)_j \log(f(x)_j) \QX(dx) \ , \label{eq:mlr-risk}
	\end{equation}
	where $y(x) \in \{0, 1\}^{n_0}$ are the one-hot encoded, noise-free labels with $y(x)_j = 1$ iff $x$ belongs to class $j$.
	In the multinomial logistic regression setting, the hypotheses are functions $f: \cX \to S_{n_0}$.
	They can be interpreted as conditional densities of a categorical distribution in the sense that the probability of a point $x \in \cX$ to belong to the class $j$ is given by ${p(y = j \mid x) = f(x)_j}$.
	
	The set $S_{n_0}$ can be turned into a Riemannian manifold of discrete probability mass functions.
	The tangent spaces given by 
	\begin{equation}
		T_p S_{n_0} = \{q \in \R^{n_0} \mid \sum_j q_j = 0\}
	\end{equation}
	will be equipped with the so called Fisher-Rao metric, given by
	\begin{equation}
		\inner{\zeta}{\xi}_p^\mathrm{FR}
		\coloneqq \sum_{j=1}^{n_0} \frac{\zeta_j \xi_j}{p_j}
		= \zeta^\T
		\underbrace{
			\diag(\frac{1}{p_1}, \ldots, \frac{1}{p_{n_0}})
		}_{\eqqcolon \cF(p)}
		\xi \ .
	\end{equation}
	Here, $\cF(p)$ is the well-known Fisher Information Matrix \cite{rao45, amari97} of a categorical distribution with parameters $p$ which is often defined as
	\begin{align}
		\cF(p)
		= \E_{y \sim p}\left[\frac{\partial \log p(y)}{\partial p} \left(\frac{\partial \log p(y)}{\partial p}\right)^\T\right] \ .
	\end{align}
	Of course, there are a priori many possible choices for the metric.
	The motivation for the Fisher-Rao metric originates in information geometry and statistics.
	This metric is (up to rescaling) the unique metric invariant under sufficient statistics.
	This fact is also known as Chentsov's theorem, see \cite[Theorem 2.6]{amari00}.

	The \emph{softmax} function $\sigma: \R^{n_0} \to S_{n_0}$ is given component-wise by
	\begin{equation}
		\sigma(x)_j
		= \frac{\exp(x_j)}{\sum_{k = 1}^{n_0} \exp(x_k)} \ .
	\end{equation}
	We choose our learning model as
	\begin{equation}
		\cS \coloneqq \sigma(H) = \{ \sigma \circ f \mid f \in \cH\} \ , \label{eq:S-definition}
	\end{equation}
	where $\cH$ is a manifold of low-rank functional TTNs mapping to $\R^{n_0}$ as described in \cref{sec:low-rank-model}, together with the parametrization $F: \cM \to \cH$, $F = G \circ \tau$. In other words, the learning model consists of functional TTNs composed with softmax. 
	
	The manifold $\cS$ is obviously parameterized by $\bar{F}: \cM \to \cS$ with
	\begin{equation}
		\bar{F}(\Theta)
		= (\sigma \circ G \circ \tau)(\Theta) 
		= \sigma(\inner{\tau(\Theta_1, \ldots, \Theta_d)}{\Phi(\cdot)}) \ .
	\end{equation}
	We hence need to compute the natural Riemannian gradient \wrt $\bar{F}$.
	Choosing the Fisher-Rao metric on $S_{n_0}$ results in the following metric for the manifold $\cS$
	\begin{equation}
		\inner{\zeta}{\xi}_s
		= \int_\cX \inner{\zeta(x)}{\xi(x)}_{s(x)}^\mathrm{FR} \QX(dx)
		= \int_\cX \zeta(x)^\T \cF(s(x)) \xi(x) \QX(dx) \ , \label{eq:fr-metric-S}
	\end{equation}
	where $s \in T_s \cS$.
	Note that $s(x) = s( \cdot \mid x)$ is a probability mass function in $S_{n_0}$.
	\cref{eq:fr-metric-S} is the Fisher-Rao metric on the space of probability densities on $\cX \times \cY$ restricted to the densities $p$ whose marginal distributions \wrt $\cX$ are $p_\cX = \mu_\cX$.

	Note that similar to the previous section, we require some regularity on the functions $s \in \cS$.
	In particular, for the existence of the integral in \cref{eq:mlr-risk} we require that each component of $s: \cX \to S_{n_0}$ is in $L_1(\cX, \QX)$ .
	This assumption again does not seem critical.
	For example, assuming all basis functions $\varphi_i^\nu$ to be continuous and integrable is sufficient.
	In particular, $\cX = \Omega_1 \times \dots \times \Omega_d \subseteq \R^d$.
	
	Summarizing the above, we compute the natural Riemannian gradient \wrt $\bar{F}$ and the metric in \cref{eq:fr-metric-S}.

	The differential of $\sigma$ is easily computed as
	\begin{equation}
		D\sigma(x) 
		= \left(D\sigma(x))_{i, j}\right) 
		= \left(\sigma(x)_i \left(\delta_{i, j} - \sigma(x)_j\right)\right) \ .
	\end{equation}
	Based on this, we provide a formula for the system matrix of in \cref{eq:ngrad}.
	\begin{proposition}
		For every $\Theta \in \cT$ it holds that
		\begin{align}
			D\bar{F}(\Theta)^\ast D\bar{F}(\Theta)
			&= \int_\cX D\tau(\Theta)^\ast 
			\left(C(f_x) \otimes \Phi_1(x_1) \Phi_1(x_1)^\T \otimes \ldots \otimes \Phi_d(x_d) \Phi_d(x_d)^\T\right) D \tau(\Theta)\,  \QX(dx)\ ,
		\end{align}
		where $f_x = G(\tau(\Theta))(x) = \inner{\tau(\Theta)}{\Phi(x)}$ and
		\begin{equation}
			C(z) = D \sigma(z)^\ast \cF(\sigma(z)) D \sigma(z) \in \R^{n_0 \times n_0} \ .
		\end{equation}
	\end{proposition}
	The proof is again an immediate consequence of \cref{prop:dfdf-as-integral}. 
	The empirical equivalent is given by
	\begin{equation} \label{eq: empirical system logistic}
		D\bar{F}(\Theta)^\ast D\bar{F}(\Theta)
		\approx \frac{1}{m}\sum_{i = 1}^m D\tau(\Theta)^\ast \left(C(f_{x^i}) \otimes \Phi_1(x_1^i) \Phi_1(x^i_1)^\T \otimes \ldots \otimes \Phi_d(x^i_d) \Phi_d(x^i_d)^\T\right) D \tau(\Theta)\ , \label{eq:logistic-empirical-ngrad}
	\end{equation}
	where $f_{x^i} = \inner{\tau(\Theta)}{\Phi(x^i)}$.

	Summarizing the model and parametrization for multinomial logistic regression, we have a chain of functions between manifolds
	\begin{equation}
		\cM \stackrel{\tau}{\longrightarrow} \cT 
		\stackrel{G}{\longrightarrow} \cH
		\stackrel{\sigma}{\longrightarrow} \cS
		\stackrel{R}{\longrightarrow} \R \ .
	\end{equation}
	We could in principle choose at which manifold the parametrization "ends", or equivalently, \wrt which (semi-)metric we want to compute the natural Riemannian gradient.
	By the above reasoning, in this application, the manifold $\cS$ with the Fisher-Rao metric is the canonical choice.
	With this choice, the natural Riemannian gradient for the multinomial logistic regression setting considered here in principle coincides with the natural gradient of \citet{amari98}, with the difference that we consider a \emph{Riemannian} gradient.
	
	\begin{remark}
		In the manner of interpreting the natural Riemannian gradient as the Riemannian gradient \wrt a certain metric, the Fisher-Rao metric can also be interpreted as the metric induced by a log-parametrization
		\begin{equation}
			 p \mapsto \log \circ\, p
		\end{equation}
		to the space of "log-densities".
		Choosing the metric $\inner{\zeta}{\xi}_{\log p} = \int_\cX \zeta(x) \xi(x) p(dx)$ on $T_{\log p} \log(S_{n_0})$, due to \cref{prop:dfdf-as-integral} leads to
		\begin{align}
			D \log(p)^\ast D\log(p) 
			&= \int_\cY D \log_y(p)^\ast D\log_y(p) p(dy)\\
			&= \sum_{j = 1}^{n_0} (0, \ldots, 0, \frac{1}{p_j}, 0, \ldots, 0)^\T (0, \ldots, 0, \frac{1}{p_j}, 0, \ldots, 0) p_j
			= \cF(p) \ ,
		\end{align}
		where $\log_y(p) = \log p(y)$.
		This means that the natural gradient \wrt the Fisher-Rao metric on $S_{n_0}$ points in the direction of steepest descent in the space of log-densities \wrt a $L_2$-type metric.
	\end{remark}

	\subsection{Further approximation strategies for practical computation}\label{sec:approximations}
	
	A central challenge for applying natural Riemannian gradient descent to tensor networks is that the operator $DF(\Theta)^\ast DF(\Theta): T_{\Theta}\cM \to T_{\Theta}\cM$ in the main linear system~\eqref{eq:ngrad} is in general extremely large and can neither be computed nor stored efficiently.
	Consider a balanced binary TTN (\cref{ex: balanced binary TTN}) with $d'$ cores and bond dimensions $\mathbf r$.
	Let $r \coloneqq \max(\mathbf{r})$ be the maximum rank.
	Then $DF(\Theta)^\ast DF(\Theta)$ is represented by a matrix of with $\mathcal{O}(r^6 (d')^2)$ entries, since each core has size $\mathcal{O}(r^3)$ and there are $d'$ cores.
	Even for rather small networks, this results in linear systems which cannot be solved efficiently (for $r = 8$, $d'=15$, the system matrix has $\approx 2.9 \times 10^7$ entries).
	In general $DF(\Theta)^\ast DF(\Theta)$ is not sparse and does not have low-rank structure.	
	We propose three strategies for approximating the operator, focusing on multinomial regression, although most ideas apply more generally.
	For convenience, these will be presented for the analytical operator, although in practice we apply the strategies to the empirical approximations~\eqref{eq:empirical-dfdf} using specific modifications for functional TTNs as discussed in the previous subsections.

	\subsubsection*{Block-diagonal approximation of $DF(\Theta)^\ast DF(\Theta)$}
	
	Instead of computing the full operator $DF(\Theta)^\ast DF(\Theta)$, we approximate it with a block diagonal approach.
	This is similar to the K-FAC as proposed in \cite{martens15}, where a block-diagonal approximation to the Fisher information matrix is computed in the setting of neural networks.
	The block diagonal approximation results in $d'$ operators acting on the core tensors individually, that is, we obtain $d'$ systems of size $\mathcal{O}(r^6)$ (instead of $\mathcal{O}(r^6 (d')^2)$) which makes natural gradient descent tractable.
	We call the resulting algorithm \textbf{BD-ngrad}, its pseudocode is provided in \cref{algo:ngd-ttn-block-diag}.
	Recall that the parameter manifold $\cM$ is a Cartesian product of manifolds, that is, $\cM = \bigtimes_{k = 1}^{d'} U_k$ for some embedded manifolds $U_k$ (see \cref{eq:M-TT,eq:M-TTN}).
	In the pseudocode,
	\begin{equation}
		D \tau_k(\Theta): T_{\Theta_k} U_k \to T_{\tau(\Theta)} \cT 
		\ , \quad
		\zeta_k \mapsto \tau(\Theta_1, \ldots, \Theta_{k-1}, \zeta_k, \Theta_{k+1}, \ldots, \Theta_{d'})
	\end{equation}
	denotes the restriction of $D \tau(\Theta)$ to the manifold $U_k$, that is, the restriction to $k$-th \mbox{(block-)}column of $D \tau(\Theta)$.
	Clearly the block-diagonal of $DF(\Theta)^\ast DF(\Theta)$ is positive semidefinite.
	Hence, the (negative) solution of the block-diagonal system is still a descent direction for $\cR \circ F$ on $\cM$ (given that the current iterate $\Theta$ is not a critical point).
	 
	Note that the function $f \in \cH$ and the system matrix $W_k$ in \cref{algo:ngd-ttn-block-diag} are in practice never explicitly computed, which is signified by "$\coloneqq$" instead of "$\leftarrow$" in the algorithm.
	Instead, we solve the linear system with a conjugate gradient method and compute evaluations of $W_k$ as needed.
	
	Since for functional TTNs, $\cM$ is embedded in an Euclidean ambient space, the Riemannian gradient $g_k$ can be easily computed by first computing the Euclidean gradient $\grad (R \circ \bar{F})(\Theta)$, either explicitly or with automated differentiation and subsequently projecting onto the respective tangent space $T_\Theta \cM$.
	
	\begin{algorithm}[t]
		\DontPrintSemicolon
		\SetAlgoLined
		\KwIn{Risk $\cR: \cS \to \R$ given in \eqref{eq:mlr-risk}, Parametrization $\bar{F} = \sigma \circ G \circ \tau: \cM \to \cS$, Initial point $\Theta^{0} \in \cM$}
		$t \leftarrow 1$  \;
		\While{not converged}{
			Compute $\rgrad (\cR \circ \bar{F})(\Theta^{(t)})$ \;
			\For{$k = 1, \ldots, d'$}{
				$g_k \leftarrow (\rgrad (R \circ \bar{F}) (\Theta^{(t)}))_k$\;
				$f \coloneqq G(\tau(\Theta^{(t)}))$\;
				$W_k \coloneqq~\frac{1}{m}\sum_{i = 1}^m D\tau_k(\Theta^{(t)})^\ast \left(C(f(x^i)) \otimes \Phi_1(x_1^i) \Phi_1(x^i_1)^\T \otimes \ldots \otimes \Phi_d(x^i_d) \Phi_d(x^i_d)^\T \right) D \tau_k(\Theta^{(t)})$\;
				\vspace{-4mm}
				$\zeta_k^{(t)} \leftarrow$ solve $W_k^{} \zeta_k^{(t)} = g_k$\;
			}	
			$\zeta^{(t)} \leftarrow (\zeta_1^{(t)}, \ldots, \zeta_{d'}^{(t)})^\T$\;
			Choose step-size $\gamma^{(t)} > 0$\;
			$\Theta^{(t+1)} \leftarrow \rR_{\Theta^{(t)}}\left(- \gamma^{(t)} \cdot \zeta^{(t)}\right)$\;
			$t \leftarrow t + 1$\;
		}
		\Return $\Theta^{(t)}$\;
		\caption{BD-ngrad for multinomial logistic regression}
		\label{algo:ngd-ttn-block-diag}
	\end{algorithm}

	\subsubsection*{Rank-one approximation of $C(z)$ for multinomial logistic regression}
	
	In multinomial logistic regression, we need to compute the matrix $C(z) = D \sigma(z)^\ast \cF(\sigma(z)) D \sigma(z)$ for each sample $x$, where $z = \inner{\tau(\Theta)}{\Phi(x)}$.
	Instead of computing the full matrix, it is possible to compute a rank-one approximation instead.
	This approach was also used in \cite{martens15arxiv}.
	Note that 
	\begin{equation}
		C(z) 
		= D \sigma(z)^\ast \cF(\sigma(z)) D \sigma(z)
		= \sum_{j = 1}^{n_0} (D \sigma(z))_{:,j} \frac{1}{\sigma(z)_j} (D \sigma(z))_{:, j}^\T \ .
	\end{equation}
	We suggest to approximate $C(z)$ with
	\begin{equation}
		C(z) \approx \frac{(D \sigma(z))_{:,k} (D \sigma(z))_{:, k}^\T}{\sigma(z)_k} \eqqcolon \tilde{C}(z) \ , \label{eq:C-approx}
	\end{equation}
	where $j \in \{1, \ldots, n_0\}$ is sampled from the probability distribution $\sigma(z)$, that is, $P(j = k) = \sigma(z)_k$.
	This reduces the computational effort further by a factor of $n_0$.
	We call this block-diagonal approach with one-shot sampling \textbf{BDO-ngrad}.
	The resulting algorithm is identical to \cref{algo:ngd-ttn-block-diag} with the sole difference that we swap $C$ for $\tilde{C}$.

	\subsubsection*{Stochastic natural Riemannian gradient descent with fully-diagonal approximation}
	
	For larger datasets it is not possible to compute the gradient with respect to all samples at once.
	Instead, we apply mini-batch stochastic gradient descent.
	Since gradients obtained from mini-batches are in general less exact estimates of the true gradient, we use momentum to stabilize the descent.
	Since we optimize on a manifold, old momenta have to be transported to the new tangent space before adding them to current gradient iterates in
	an exponential decay averaging fashion according to a decay parameter $\beta_1$.
	This can be done through a transport map $\mathrm{T}_{\Theta_2 \leftarrow \Theta_1}: T_{\Theta_1} \cM \to T_{\Theta_2} \cM$ which projects the previous gradient onto the tangent space of the new point $\Theta_2$; for more details, see \eg \cite{boumal23}.
	
	Similarly, estimating the operator $DF(\Theta)^\ast DF(\Theta)$ from only a few samples is not expected to produce a good approximation.
	We suggest to use momentum here, too, again using exponential averaging with decay parameter $\beta_2$.
	Like in Riemannian BFGS this would entail transporting matrices to the new tangent space, too.
	While possible in theory, this can be difficult and inefficient in practice.
	In order to avoid the transport altogether, we propose a fully-diagonal approximation of the operator by approximating the $(k, k)$-block of $DF(\Theta)^\ast DF(\Theta)$ in the following way:
	\begin{equation}
		\left(DF(\Theta)^\ast DF(\Theta)\right)_{k, k} \approx \lambda^\mathrm{max}_k I \ ,
	\end{equation}
	where $\lambda^\mathrm{max}_k$ is the maximal eigenvalue of $\left(DF(\Theta)^\ast DF(\Theta)\right)_{k, k}$.
	Note that again the (negative) solution of the fully-diagonal system is a descent direction since the fully-diagonal system is positive definite.
	
	The obvious benefit of this fully-diagonal approximation is that we do not have to solve the natural gradient system \cref{eq:ngrad} with conjugate gradients, instead, the solution is given by simply dividing the right hand side by $\lambda^\mathrm{max}_k$.
	The eigenvalue $\lambda^\mathrm{max}_k$ itself can be estimated through power iteration.
	Numerical experiments (see \cref{sec:experiments}) show that even just a single iteration of the power method, starting from the current gradient as the initial point, is enough to improve learning.
	This overall leads to a significant speed-up as well.
	In practice, we combine this fully-diagonal approximation with the rank-one of $C(z)$ approximation from above.
	We call the combined algorithm \textbf{D-ngrad}, pseudocode is depicted in \cref{algo:ngd-d-ngrad}.

	\begin{algorithm}[t]
		\DontPrintSemicolon
		\SetAlgoLined
		\KwIn{Risk $\cR: \cS \to \R$ given in \eqref{eq:mlr-risk}, Parametrization $\bar{F} = \sigma \circ G \circ \tau: \cM \to \cS$, Initial point $\Theta^{0} \in \cM$, $\beta_1, \beta_2 \in (0, 1]$}
		$\lambda_k^{(0)} \leftarrow 1$ for $k = 1, \ldots, d'$\;
		$g_k^{(0)} \leftarrow 0$ for $k = 1, \ldots, d'$\;
		$t \leftarrow 1$\;
		\While{not converged}{
			\For{batch $B$ in batches}{
				Compute $\rgrad (\cR \circ \bar{F})(\Theta^{(t)})$ \;
				\For{$k = 1, \ldots, d'$}{
					$g_k \leftarrow (\rgrad (R \circ \bar{F}) (\Theta^{(t)}))_k$\;
					$f \coloneqq \bar{F}(\Theta^{(t)})$\;
					$W_k \coloneqq~\frac{1}{\abs{B}}\sum_{i \in B} D\tau_k(\Theta^{(t)})^\ast \left(\tilde{C}(f(x^i)) \otimes \Phi_1(x_1^i) \Phi_1(x^i_1)^\T \otimes \ldots \otimes \Phi_d(x^i_d) \Phi_d(x^i_d)^\T \right) D \tau_k(\Theta^{(t)})$\;
					\vspace{-5mm}
					$\lambda_k \leftarrow \frac{g_k^\T W_k g_k}{g_k^\T g_k}$\;
					$\lambda^{(t)}_k \leftarrow \beta_2 \lambda^{(t-1)}_k + (1 - \beta_2) \lambda_k$\;
					$\zeta_k^{(t)} \leftarrow \beta_1 \zeta_k^{(t-1)} + (1-\beta_1) \frac{g_k}{\lambda^{(t+1)}_k}$ \;
				}
				$\zeta^{(t)} \leftarrow (\zeta_1^{(t)}, \ldots, \zeta_{d'}^{(t)})^\T$\;
				Choose step-size $\gamma^{(t)} > 0$\;
				$\Theta^{(t+1)} \leftarrow \rR_{\Theta^{(t)}}\left(- \gamma^{(t)} \cdot \zeta^{(t)}\right)$\;
				$\zeta^{(t)} \leftarrow \mathrm{T}_{\Theta^{(t+1)} \leftarrow \Theta^{(t)}}(\zeta^{(t)})$\;
				$t \leftarrow t + 1$\;
			}
		}
		\Return $\Theta^{(t)}$\;
		\caption{D-ngrad}
		\label{algo:ngd-d-ngrad}
	\end{algorithm}

	While this approximation of $DF(\Theta)^\ast DF(\Theta)$ may seem extremely rough and to show little resemblance with the original derivation of the natural Riemannian gradient, we argue it is still derived from it in a systematic way.
	Note that the algorithm is similar, but not identical to the ADAM~\cite{kingma17} and Fisher ADAM~\cite{hwang24} algorithms. 
	These other approaches employ the so-called empirical Fisher information matrix, which importantly is no empirical approximation in the sense of \cref{eq:empirical linear equation}, although somewhat related to it (see e.g. \cite[Section 11]{martens20}).

	\section{Numerical experiments}
	\label{sec:experiments}
	In this section, we discuss implementation details and numerical experiments.

	\subsubsection*{Implementation details}

	For our experiments we use the balanced binary TTN format as described in \cref{ex: balanced binary TTN}, working directly with the TTN parameter space $\cM$,
	as opposed to the embedded manifold of low-rank tensors $\cT$.
	The evaluations of model responses and Riemannian gradients on $\cM$, which are required both for baseline algorithms as well as natural gradients,
	is done through forward/backpropagation on the tree tensor network, which are worked out in detail in~\cite[Sec.~6.2]{willner25}. Concretely, forward propagation calculates the empirical risk
	\begin{equation}
		(\widehat{\cR} \circ F)(\Theta) = \frac{1}{m} \sum_{i = 1}^m  \left(I_{n_0} \otimes \Phi_1(x^i_1)^T \otimes \cdots \otimes \Phi_d(x^i_d)^T \right) \tau(\Theta) \, .
	\end{equation}
	Naively, this would require the evaluation of $m$ high-dimensional tensor-vector products, but by vectorizing and leveraging the binary tree structure,
	it can be achieved through a recursion of MTTKRP (matricized-tensor-times-Khatri-Rao-product, coined by \citet{bader08})
	operations acting on the individual cores of the network. Similarly, through a series of MTTKRP,
	backpropagation evaluates
	\begin{equation} \label{eq: backprop}
		\rgrad (\widehat{\cR} \circ F)(\Theta)
		= \frac{1}{m} \sum_{i = 1}^m  D\tau(\Theta)^\ast \left[v^i \otimes \Phi_1(x^i_1) \otimes \cdots \otimes \Phi_d(x^i_d) \right] \, ,
	\end{equation}
	where $v^i = \nabla \cL_{y_i}(F(\Theta)(x_i))$ for $\cL_{y}: \R^{n_0} \to \R$ with $\cL_{y}(h(x)) = \ell(h, x, y)$. 
	Those workloads, as well as other computationally intensive subtasks are handled in parallel through a dedicated \texttt{C++} library
	This library interfaces with a \texttt{Python} library which acts on higher abstraction levels, implementing the TTN network architecture and the presented descent algorithms themselves.

	A comment about the computation the empirical system matrices in \eqref{eq: empirical system l2} and \eqref{eq: empirical system logistic} is in order.
	They both involve diagonal matrices, $I_{n_0}$ and $C(z^i)$ respectively.
	We denote these matrices by $\Delta^i \in \R^{n_0 \times n_0}$ and their diagonal entries $\Delta^i_j$, $j = 1, \ldots, n_0$. 
	Writing
	\begin{equation}
		\xi^i_j = D\tau(\Theta)^\ast \left[ e_j \sqrt{\Delta^i_{j}} \otimes \Phi_1(x_1^i) \otimes \ldots \otimes \Phi_d(x^i_d) \right]
	\end{equation}
	with $e_j$ the $j$-th Cartesian unit vector, the empirical version of the system matrix reads
	\begin{equation}
		DF(\Theta)^\ast DF(\Theta) \approx \sum_{j=1}^{n_0} \frac{1}{m}\sum_{i = 1}^m  \xi^i_j (\xi^i_j )^T.
	\end{equation}
	By setting $v^i = e_j \sqrt{\Delta^i_{j}}$ the vectors $\xi^i_j$ can be efficiently evaluated with an additional backpropagation step~\eqref{eq: backprop}.
	All of our algorithms employ this approach of representing the full system matrix as the above sum of rank-one terms.
	Of course, the full matrix is never computed explicitly.
	For solving the the linear system we instead use the matrix-free CG solver.
	Note that for the BD-ngrad approximation, only the block-diagonal parts of the outer products have to be considered.
	For the BDO-ngrad approximation the sum over $j$ additionally reduces to a single term.
	In D-ngrad the eigenvalues of the blocks are estimated by directly applying the power method in the rank-one decomposition format.

	In our implementation both standard and natural Riemannian gradients are further subjugated to the quotient manifold formalism described in~\cref{rem: quotient formalism}.
	In particular, we employ the Cartesian horizontal space $H^\equiv \cM$.
	This means solving an empirical version of the system in \cref{eq:p-equiv-ngrad-system}, that is
	\begin{equation}
		P_\Theta^\equiv \left(\frac{1}{m} \sum_{i = 1}^m DF_{x^i}(\Theta)^\ast DF_{x^i}(\Theta)\right) P_\Theta^\equiv[\zeta]
		= P_\Theta^\equiv \left(\frac{1}{m} \sum_{i=1}^m \rgrad \ell(F(\cdot), x^i, y^i)(\Theta)\right) \ .
	\end{equation}
	The operator on the left hand side need not always have full rank.
	In our experiments, we almost never observed the operator on the left hand side to be of full rank.
	This is because the number of samples employed were too small to achieve full rank in the above rank-one formulation: $m < \dim(M) = \dim(\text{range}(DF(\Theta)^\ast))$.
	Therefore we regularize the system according to
	\begin{equation}
		P_\Theta^\equiv DF(\Theta)^\ast DF(\Theta) P_\Theta^\equiv
		\approx P_\Theta^\equiv \left(\frac{1}{m} \sum_{i = 1}^m DF_{x^i}(\Theta)^\ast DF_{x^i}(\Theta) + \lambda I \right) P_\Theta^\equiv \ .
	\end{equation}
	In all experiments we choose the regularization parameter as $\lambda = 5 \times 10^{-3}$ and solve the system
	on $H_\Theta^\equiv \cM$ using conjugate gradients,
	since the regularized operator is symmetric positive definite on the horizontal space.

	Since the computed natural Riemannian gradients are in $H_\Theta^\equiv \cM$, we can use the computationally inexpensive QR-based retraction described in~\cite[Sec. 4.4]{willner25}.
	A suitable vector transport map for~$\cM$, which is also compatible with the quotient structure is given by
	\begin{equation}
		\mathrm{T}_{\Theta_2 \leftarrow \Theta_1}: H_{\Theta_1}^\equiv \cM \to H_{\Theta_2}^\equiv \cM\ , \quad
		\mathrm{T}_{\Theta_2 \leftarrow \Theta_1} = P_{\Theta_2}^\equiv \ ,
	\end{equation}
	that is, the projection onto to the Cartesian horizontal space at $\Theta_2$ (see \eg \cite[Ex.~10.67]{boumal23}).
	We use this construction to transport momentum vectors to new iterates for the stochastic experiments.
	All experiments were conducted on an \texttt{Intel Ultra 7 155H} CPU complemented with 32GB of DDR5 RAM.

	\subsection{Least-squares recovery problem}\label{sec:Least-squares recovery problem}

	As a basic proof-of-concept, we consider a simple TTN recovery problem with least-squares loss.
	As training objective, we generate a randomly initialized balanced binary TTN from \cref{ex: balanced binary TTN} with $3$ cores as parameter $\Theta \in \cM$, TTN-ranks $\mathbf r =(5, 5)$ and output dimension $n_0 = 3$.
	The resulting TTN therefore describes a function $h^\ast = F(\Theta): \R^4 \to \R^3$.
	For the bases, we choose monomial bases of order $2$, which means that $h^\ast_j \in \cV = \R[x]_2 \otimes \R[x]_2 \otimes \R[x]_2 \otimes \R[x]_2$ for $j=1,2,3,4$.
	The training target is to recover the function $h^\ast$ from noisy training data.

	The inputs of the training data ($m = 256$) are generated by sampling uniformly from $[-1, 1]^4$, that is, we have $\QX(X) = \prod_{i=1}^4 \mu^{\textrm{uniform}}_{[-1, 1]}(X_i)$ for $X = X_1 \times X_2 \times X_3 \times X_4$.
	The targets are chosen as $y^i = h^\ast(x^i) + \epsilon^i$, where the components of $\epsilon^i \in \R^3$ are sampled from a centered Gaussian distribution with variance $\sigma^2 = 2.5 \times 10^{-3}$.
	A perfect recovery of $h^\ast$ would therefore result in an expected training loss of about $3 \sigma^2 = 7.5 \times 10^{-3}$.

	As the starting iterate for the experiments, we take a second randomly generated TTN $h_0: \R^4 \to \R^3$ of the similar type with parameter $\Theta_0 \in \cM$.
	For each individual experiment we choose identical basis vectors $\Phi_\nu$ for all $\nu=1, \ldots, 4$, but vary the concrete basis of $\R[x]_2$.
	In particular, we conduct experiments using a monomial basis, a normalized Legendre basis (which is an ONB of $\R[x]_2 \otimes \R[x]_2 \otimes \R[x]_2 \otimes \R[x]_2$ in $L_2(\R^4, \R; \QX)$) and a Hermite basis.
	For the different bases the initial point $\Theta_0$ is rescaled such that it always represents the same~$h_0$.
	In the experiments we compare standard Riemannian gradient descent (grad) and natural Riemannian gradient descent (ngrad) for the different bases.
	For choosing step sizes, we employ a two-way backtracking line search according to the Armijo-Goldstein criterion, see, e.g. \cite[Section 4.5]{boumal23}.

	Results for one instance are displayed in~\cref{fig:change_of_basis}.
	It can be observed that the natural gradient methods react less sensitive to a change of basis than standard gradient methods, always outperforming their respective counterpart.
	While in theory the natural Riemannian gradient is independent of the choice of basis, this cannot exactly be observed in practice.
	Possible reasons for this are discussed at the end of \cref{sec:low-rank-model}.
	Out of the standard Riemannian gradient methods, the Legendre basis descents fastest.
	This is not surprising, since the orthonormal Legendre polynomials are optimally conditioned w.r.t. $\QX$, which implies that at least the function $G$ in the parametrization $F = G \circ \tau$ is an isometry and should also result in a well-conditioned $DF(\Theta)$,
	cf.~\cref{remark: meaure product structure}.

	\begin{figure}[t]
		\begin{center}
			\includesvg[width=0.5\textwidth]{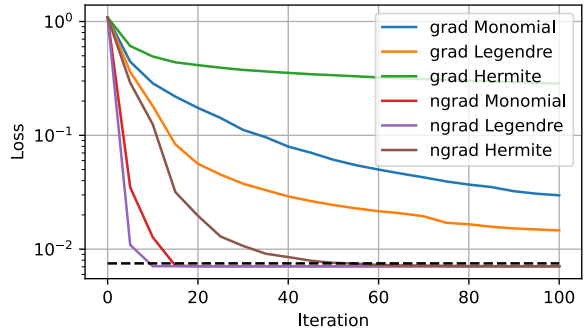}
			\caption{Comparison of grad and ngrad methods for a recovery problem under change of basis. The natural Riemannian gradient methods achieve the expected minimum loss (black line).}
			\label{fig:change_of_basis}
		\end{center}
	\end{figure}

	\subsection{Deterministic multinomial logistic regression}
	\label{sec:deterministic-logistic-regression}

	We evaluate the deterministic algorithms in the multinomial classification setting on the \texttt{digits} dataset \cite{alpaydin98}, which consists of $m=1726$ grayscale images of hand-written digits that are to be classified into $n_0 = 10$ classes.
	Each image consists of $8\times 8$ pixels, so we pick $d=64$.
	We employ a $80/20$ train-test split.
	The target labels of the training data are encoded as one-hot vectors in $\R^{10}$.
	We choose the basis vectors $\Phi_\nu(x) = \frac{(1, x)^T}{\norm{(1, x)}}$ for all $\nu = 1, \ldots, d$, which we adapted from \cite{stoudenmir18}, and can be interpreted as a affine linear basis with normalization.
	We observed that this basis works well in practice, compared to unnormalized bases.
	A suitable starting iterate and TTN-ranks $\mathbf{r}$ were found by using the unsupervised coarse-graining method proposed by \citet{stoudenmir18}, with maximum rank $8$.

	\begin{figure}[t]
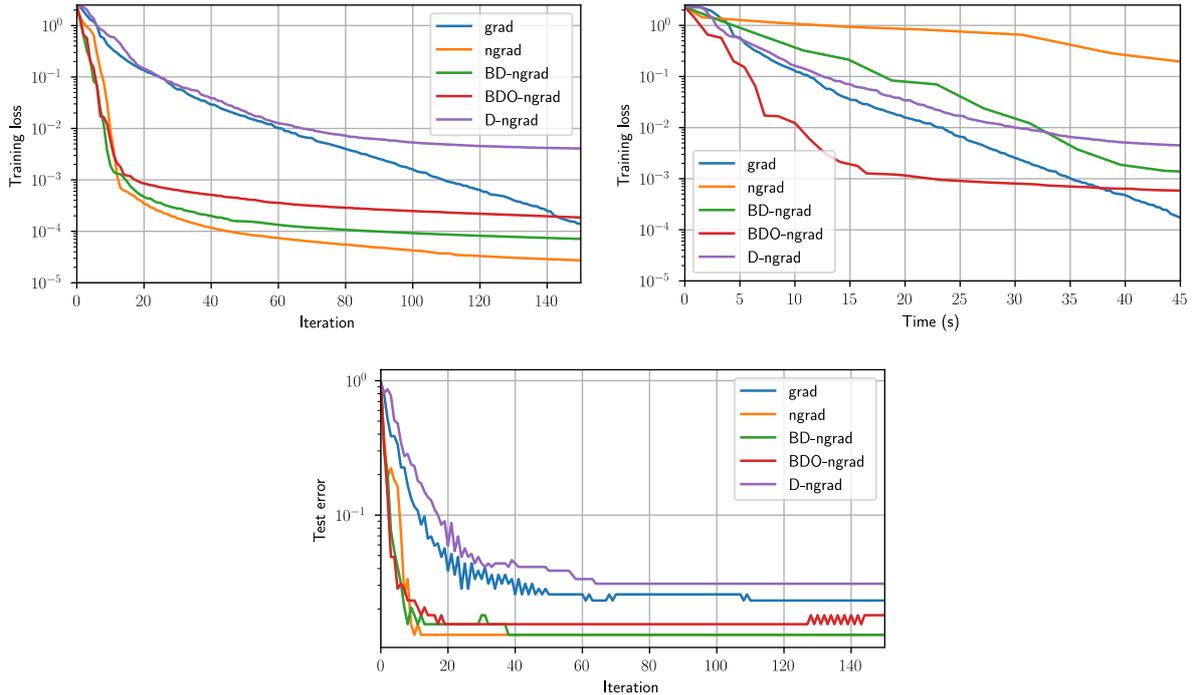

		\centering
		\begin{minipage}[t]{0.5\textwidth}
			\centering
			\includesvg[width=\linewidth]{loss.svg}
		\end{minipage}%
		\begin{minipage}[t]{0.5\textwidth}
			\centering
			\includesvg[width=\linewidth]{performance.svg}
		\end{minipage}

		\begin{minipage}{0.5\textwidth}
			\centering
			\includesvg[width=\linewidth]{accuracy.svg}
		\end{minipage}
		\caption{Comparison of standard Riemannian gradient descent (grad), natural Riemannian gradient descent (ngrad), BD-ngrad, BDO-ngrad and D-ngrad for the \texttt{digits} dataset.}
		\label{fig:digits}
	\end{figure}

	In the experiments we compare standard Riemannian gradient descent, natural Riemannian gradient descent, BD-ngrad and BDO-ngrad.
	For completeness, we also compare a non-stochastic variant of D-ngrad, where we only have a single batch of size $m$ and only use momentum for the eigenvalues but not for the gradient (which also means there is no transport of gradients).
	Decay parameters were chosen $\beta_1 = 0$ and $\beta_2 = 0.9$.
	In all algorithms, we use a two-way backtracking line search according to the Armijo-Goldstein criterion to choose step sizes.

	Results of our experiments are displayed in \Cref{fig:digits}.
	The top row shows training loss plotted against the number of iterations (left) and time (right).
	As expected, the proposed hierarchy of approximations subsequently reduces computational effort at the cost of deteriorating convergence.
	It can be observed that gradient descent overtakes the natural gradient methods at some point (top-right plot).
	Note however that this happens only when the natural gradient methods have already converged in terms of test accuracy (as seen in the bottom-right figure) and that this makes little to no difference in the final test accuracies after 500 iterations, which are reported in \cref{tab:accuracies_digits}.
	\begin{table}[t]
		\centering
		\begin{tabular}{|c|c|c|c|c|}
			\hline
			grad & ngrad & BD-ngrad & BDO-ngrad & D-ngrad\\
			\hline
			98.71\% & 98.71 \% & 98.71 \% & 98.20 \% & 97.17 \% \\ \hline
		\end{tabular}
		\caption{Final test accuracies on \texttt{digits} after 500 iterations}
		\label{tab:accuracies_digits}
	\end{table}

	\subsection{Stochastic multinomial logistic regression}
	\label{sec:deterministic-logistic-regression}

	In order to investigate the stochastic setting, we conduct experiments for the larger \texttt{MNIST} dataset \cite{lecun98}.
	This dataset consists of pictures of handwritten digits that are again to be classified into on of $n_0 = 10$ classes.
	The test setup is mostly identical to that in \cref{sec:deterministic-logistic-regression}; we only highlight differences here.
	\texttt{MNIST} pictures have a resolution of $28 \times 28$ pixels, which would require an unbalanced binary TTN.
	Although unbalanced trees are both theoretically and practically viable, we scale down the samples to $16 \times 16$, allowing the use of a balanced tree with $d = 256$.
	We employ the same feature map as in \cref{sec:deterministic-logistic-regression} and again use the unsupervised coarse-graining method \cite{stoudenmir18} for initialization, this time with maximum tree tensor ranks of 16.
	Again, we employ a random 80/20 train-test-split.

	In our experiments, we compare stochastic Riemannian gradient descent with momentum (grad) and D-ngrad.
	For both algorithms we use batch sizes of $128$ and fixed step sizes $\gamma = 16$ for grad and $\gamma = 4$ for D-ngrad, which were found using a grid search.
	The momentum decay parameters were chosen as $\beta_1 = \beta_2 = 0.9$.

	The plots in \cref{fig:mnist} compare the training loss of grad and D-ngrad for this setup.
	It can be observed that the natural gradient method outperforms the classical approach, both in terms of number of iterations and runtime.
	Furthermore, D-ngrad also achieves a better qualitative result:
	After $1000$ iterations, the final test accuracies are $87.13\%$ for grad and $96.64\%$ for D-ngrad.

	\begin{figure}[t]
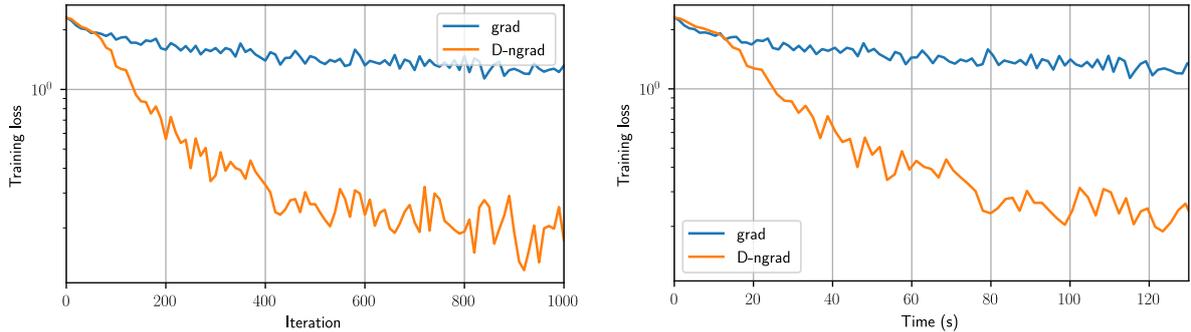

		\centering
		\begin{minipage}[t]{0.5\textwidth}
			\centering
			\includesvg[width=\linewidth]{loss_stochastic.svg}
		\end{minipage}%
		\begin{minipage}[t]{0.5\textwidth}
			\centering
			\includesvg[width=\linewidth]{performance_stochastic.svg}
		\end{minipage}
		\caption{Comparison of stochastic Riemannian gradient descent (grad) and D-ngrad for the \texttt{MNIST} dataset.}
		\label{fig:mnist}
	\end{figure}

	\section{Outlook}
	In this work we applied the concept of a natural gradient to machine learning tasks with functional TTNs as the learning model.
	We derived formulas for computing the natural Riemannian gradient both for least-squares regression and multinomial logistic regression and proposed several approximations to the natural gradient that lead to efficient optimization algorithms.
	The convergence of these methods, depending on the level of approximation, is still an open question.
	
	Since our algorithms all work with \emph{fixed} manifolds $\cM$, $\cT$ and $\cH$, choosing and fixing the bond dimensions of the TTN is required a priori, \ie, when designing the model and in particular \emph{before} optimization.
	However, it is not clear how to best choose the bond dimensions of the network to achieve a given loss or accuracy.
	Ideally, bond dimensions would be chosen automatically during optimization, which is, however, not directly possible with the algorithms suggested in this work.
	The design of a rank-adaptive algorithm based on natural Riemannian gradient descent is left for future research.
	
	Functional tensor networks can in theory also be composed as layers to form larger and potentially more expressive models.
	For functional tensor trains, such a compositional model was considered in \cite{schneideroster24,eigel25}.
	The ideas for approximating the natural Riemannian gradient presented in this work could also be useful in a compositional functional (tree) tensor framework and lead to more efficient optimization algorithms.
	However, this is left for future research.
	
	\subsubsection*{Acknowledgments}
	
	The authors would like to thank Timo Felser and Tensor AI Solutions for providing the code framework that allowed the numerical evaluation of our findings.
	The work of A.U.~was supported by the Deutsche Forschungs\-gemeinschaft (DFG, German Research Foundation) – Projektnummer 506561557.

	\renewcommand{\bibname}{References}
	\bibliographystyle{plainnat}
	\bibliography{bibliography}
	
\end{document}